\documentclass{amsart}
\usepackage{amssymb, amsmath, amsthm, amsfonts, amscd}
\usepackage{xcolor}
\usepackage{datetime}
\usepackage{enumerate}
\usepackage{mathabx}
\usepackage{blkarray}

\newtheorem{theorem}{Theorem}[section]
\newtheorem{proposition}[theorem]{Proposition}
\newtheorem{lemma}[theorem]{Lemma}

\theoremstyle{definition}
\newtheorem{definition}[theorem]{Definition}
\newtheorem{example}[theorem]{Example}

\theoremstyle{remark}
\newtheorem{remark}[theorem]{Remark}

\numberwithin{equation}{section}

\newcommand{\clm}{\mathcal{M}}
\newcommand{\clh}{\mathcal{H}}
\newcommand{\cld}{\mathcal{D}}
\newcommand{\clk}{\mathcal{K}}
\newcommand{\clb}{\mathcal{B}}

\newcommand{\clf}{\mathcal{F}}

\begin{document}

\begin{center}
{ 
       {\Large \textbf { \sc  Factorization of Characteristic Functions of Iterated Liftings                                  }
       }}
\end{center}

\vspace{1cm}

\begin{center} 
Neeru Bala\\
Stat-Math Unit,\\ 
Indian Statistical Institute,\\ 
R. V. College Post, Bangalore- 560059, India\\
Email address:  neeru$\_$vs@isibang.ac.in

\vspace{.5cm}
Santanu Dey\\
 (corresponding author)\\
 Department of Mathematics,\\
  Indian Institute of Technology Bombay,\\
   Powai, Mumbai-400076, India\\
Email address: santanudey@iitb.ac.in

\vspace{.5cm}
M. N. Reshmi\\
 Department of Mathematics,\\ 
 Indian Institute of Technology Bombay,\\
  Powai, Mumbai-400076, India\\
Email address: reshmi@math.iitb.ac.in
\end{center}

\vspace{1cm}

\vspace{0.5cm} {\bf Abstract:}  {We obtain a factorization of the characteristic function of a contractive two-step iterated lifting in terms of the characteristic functions of constituent liftings of the iterated lifting and the Julia-Halmos matrix. We also give an expression for the characteristic function of the minimal part of a contractive two-step iterated lifting as a restriction of the product of the characteristic functions of constituent liftings of the iterated lifting.}

\vspace{0.5cm}
\noindent {\bf keywords:} row contractions, contractive lifting, Fock space, characteristic functions, multi-analytic operators, completely non-coisometric

\vspace{0.5cm}

	\section{Introduction}
	The aim of this article is to study lifting of row contractions on a Hilbert space. The foundation for the study of contractions on Hilbert spaces using dilation theory was laid down by  Sz-Nagy, C. Foias, Sarason, etc. (\cite{NF1}).  Let $\{S_i\}_{i=1}^d$ be a set of bounded linear operators on a Hilbert space $\mathcal{H}$. Then $\underline{S}=(S_1,S_2,\ldots,S_d)$ is called a {\it row contraction}, if $\underset{i=1}{\overset{d}{\sum}}S_iS_i^*\leq I_{\mathcal{H}}$. The work by Sz-Nagy and C. Foias \cite{NF1} is extended to different setups, for example, to row contractions by G. Popescu \cite{GP2,GP1,POP,POP2,POP3}, to commutative row contraction by W. Arveson \cite{ARV} and to a pair of commuting contractions by Ando \cite{{AND}}. 
	
	To every contraction on a Hilbert space, Sz-Nagy and C. Foias \cite{NF1} adjoined an operator valued analytic function, called a characteristic function, which is a completely invariant upto unitary equivalence for the contractions. G. Popescu \cite{GP2}  introduced characteristic function for (non commutative) row contractions and extended Sz-Nagy-Foias theory of characteristic function. Characteristic functions of commuting row contractions are studied in  \cite{BES,BT} and we refer to references therein for more details. Later, S. Dey and R. Gohm (\cite{DEY,DGH15,DEY1}) introduced characteristic functions for contractive liftings and established that it is a unitary invariant for a large class of liftings.  Popescu's theory of characteristic function of row contractions was shown as a special case of the theory of characteristic function of lifting in \cite{DGH15}. In this article, we study factorizations of the characteristic function of contractive liftings defined in \cite{DEY}. 
	
	In this article, one of the key element is Julia-Halmos matrix. Julia-Halmos matrix  corresponding to a bounded linear operator $L$ on a Hilbert space $\mathcal{H}$ is a unitary matrix of the form 
	\[
	J_L = \begin{pmatrix}
		D_{*,L} & L \\
		-L^{*} & D_L
	\end{pmatrix}
	\] where $D_L$ and $D_{*,L}$ denote the defect operators of $L.$  Julia-Halmos matrix plays a significant role in dilation theory and we refer to \cite{TIM} for more details about the matrix.
	
There are many factorizations of characteristic function available in the literature. Among them the most significant one is inner-outer factorization \cite{NF1}. A considerable amount of research work has been published on inner-outer factorization and some of the references for this are \cite{NF1,GP2}. A more general factorization of a contractive analytic function is the  regular factorization (\cite[Page 297]{NF1}).  The inner-outer factorization is also a regular factorization. Moreover, regular factorization is connected to another very interesting area in operator theory, that is, invariant subspaces of the contraction. Corresponding to every non-trivial invariant subspace of  a contraction, there is  a non-trivial regular factorization of the characteristic function. On other hand, if a given non-trivial factorization of the characteristic function is regular, then there exist a non-trivial invariant subspace for the contraction. G. Popescu \cite{GP1} defined regular factorization for row contraction and established a one-one correspondence between joint invariant subspace and regular factorization.

 Besides regular factorization, Sz-Nagy and C. Foias \cite{NF} also demonstrated a factorization technique using Julia-Halmos matrix, where regularity condition of the characteristic function is not required. As in regular factorization, the factorization  in \cite{NF} for the characteristic function of the contraction \[
	 	T = \begin{pmatrix}
	 		T_1 & X\\
	 		0 &T_2
	 	\end{pmatrix}
	 	\]
is achieved in terms of the characteristic function of operators $T_1$ and $T_2.$ However the proof is more direct and does not use functional model.  K. J. Haria et. al. \cite{MAJI} extended this notion to characteristic functions of non commutative row contractions. In this article, we prove a similar factorization for characteristic functions of contractive two-step iterated liftings.

	 In section $2$ of this article, we introduce notations and recall results related to row contractions which are used in this article. In section $3$, we give a Nagy-Foias type matrix factorization of characteristic functions for minimal contractive iterated liftings based on Julia-Halmos matrix. Let $\underline{E'}$ be a contractive iterated lifting consisting of two minimal liftings and let $\underline{\tilde{E}}$ be the minimal part of the iterated lifting.  In section $4$, we prove that the characteristic function of the minimal part $\underline{\tilde{E}}$ of the iterated lifting $\underline{E'}$ is the restriction to the defect space of $\underline{\tilde{E}}$ of product of characteristic functions of these minimal liftings.

	\section{Preliminaries}
	For $d\geq 2$, the full Fock space over $\mathbb{C}^d$ is defined by
	$$\Gamma(\mathbb{C}^d):=\mathbb{C}\oplus\mathbb{C}^d\oplus(\mathbb{C}^d)^{\otimes2}\oplus\cdots (\mathbb{C}^d)^{\otimes m}\cdots.$$
	In this article, we denote $\Gamma(\mathbb{C}^d)$ by $\Gamma$. The standard orthonormal basis of $\mathbb{C}^d$ is denoted by $\{e_1,e_2,\ldots, e_d\}$ and $e_0:=1\oplus 0\oplus\cdots$ is the vacuum vector. Let $\Lambda$ be the set $\{1,2,\ldots, d\}$ and $\tilde{\Lambda}=\underset{n=0}{\overset{\infty}{\bigcup}}\Lambda^n$, where $\Lambda^0:=\{0\}$. For $\alpha=(\alpha_1,\alpha_2,\ldots,\alpha_n)\in \Lambda^n$, we define  $|\alpha|:=\alpha_1+\alpha_2+\cdots+\alpha_n$ and  $e_{\alpha}:=e_{\alpha_1}\otimes e_{\alpha_2}\otimes\cdots\otimes e_{\alpha_n}.$  Note that the set $\{e_{\alpha}:\alpha\in\tilde{\Lambda}\}$ forms an orthonormal basis for the full Fock space $\Gamma$. The operators $L_i$ on $\Gamma,$ defined by $L_i h= e_i \otimes h$ for $i=1, \ldots, d$ and $h\in \Gamma,$ are called the (left) creation operators.
	
Let $\underline{S}=(S_1,S_2,\ldots,S_d)$ be a row contraction on  $\mathcal{H}_S$. Then the {\it defect operators }and {\it defect spaces }are defined by
\begin{align*}
 	D_S=\left(\delta_{ij}I-S_i^*S_j\right)^{1/2},\,\,\,\,&D_{*,S}=\left(I-\underset{i=1}{\overset{d}{\sum}}S_iS_i^*\right)^{1/2}\text{ and }\\
 	\mathcal{D}_S=\overline{\text{range }D_S},\,\,\,\,&\mathcal{D}_{*,S}=\overline{\text{range }D_{*,S}},
 \end{align*}respectively. For $\alpha=(\alpha_1,\alpha_2,\ldots,\alpha_n)\in \Lambda^n$,  we set $S_{\alpha}:=S_{\alpha_1}S_{\alpha_2}\ldots S_{\alpha_n}$.
For a row contraction $\underline{S}$ on $\mathcal{H}_S$, Popescu, in \cite{POP ML}, obtained a realization of the minimal isometric dilation (mid in short) $\underline{V}^S$ of $\underline{S}$ on $\hat{\clh}_S=\clh_S\oplus (
\Gamma\otimes \mathcal{D}_S)$ as
\begin{equation*}
	V_j^S\left( h\oplus\underset{\alpha\in\tilde{\Lambda}}{\sum}e_{\alpha}\otimes d_{\alpha} \right)= S_jh\oplus\left(e_0\otimes(D_S)_jh\oplus e_j\otimes\underset{\alpha\in\tilde{\Lambda}}{\sum}e_{\alpha}\otimes d_{\alpha}\right),
\end{equation*} 
for $h\in \clh_S,\,d_{\alpha}\in\cld_S$, where $(\cld_S)_jh=\cld_S(0,\ldots,h,\ldots,0).$

Let $\underline{C}=(C_1,C_2,\ldots,C_d)$ be a row contraction on $\mathcal{H}_C$. Then $\underline{E}$ is called a {\it lifting} of $\underline{C}$ by $\underline{A}$ on $\mathcal{H}_E=\mathcal{H}_C\oplus \mathcal{H}_A$, if 
	\[E_i=\begin{bmatrix}
		C_i&0\\
		B_i&A_i
	\end{bmatrix}\text{ for }i=1,2,\ldots,d.\]
Throughout this article, we work with contractive liftings which are characterized by the following result.
	\begin{proposition}\cite[Proposition 3.1]{DEY}\label{Prop lifting}
		Let $\underline{E}=(E_1,E_2,\ldots,E_d)$, where  $E_i$ be a bounded linear operator on $\mathcal{H}_E=\mathcal{H}_C\oplus \mathcal{H}_A$ with block matrix representation
		\[E_i=\begin{bmatrix}
			C_i&0\\
			B_i&A_i
		\end{bmatrix}\text{ for }i=1,2,\ldots,d.\] Then $\underline{E}$ is a row contraction if and only if $\underline{C}$ and $\underline{A}$ are row contractions and there exists a contraction $\gamma:\mathcal{D}_{*,A}\rightarrow \mathcal{D}_C$ such that $\underline{B}^*=D_C\gamma D_{*,A}.$
	\end{proposition} 

The contraction $\gamma:\cld_{*,A}\rightarrow \cld_C$ in Proposition \ref{Prop lifting} is called {\it resolving} \cite[Definition 3.2]{DEY}, if for all $h\in \mathcal{H}_A$ we have the following:\\ 
$$(\gamma D_{*,A}A_{\alpha}^*h=0 \mbox{~for all~} \alpha\in\tilde{\Lambda}) \implies   ( D_{*,A}A_{\alpha}^*h=0 \mbox{~for all~} \alpha\in{\tilde{\Lambda}}).$$
A row contraction $\underline{T}=(T_1,\ldots,d)$ on a Hilbert space $H$ is said to be {\it completely non-coisometric} (cf. \cite{GP2}) if 
\[  \{h\in H: \sum_{|\alpha|=n}  \|T^*_\alpha h\|^2 = \|h\|^2  \mbox{~for all~} n \in \mathbb{N} \}=\{0\}  \]

\begin{definition}\cite[Definition 3.4]{DEY}
A lifting $\underline{E}$ of $\underline{C}$ by $\underline{A}$ is called {\it reduced}, if $\underline{A}$ is completely non-coisometric and $\gamma$ is resolving. 
\end{definition}

	\begin{definition}
	A contractive lifting $\underline{E}$ of $\underline{C}$ is called minimal, if 
	\begin{align*}
		\clh_{E} = \overline{span}\{E_{\alpha}x: x\in\clh_C\;\;\text{for all }\,\alpha\in\tilde{\Lambda}\}.
	\end{align*}
\end{definition}
As noted  in the remark that follows \cite[Proposition 3.8]{DGH15} (also see \cite{DEY1}), reduced liftings are same as the {\it minimal contractive liftings}.  For more detail, we refer to \cite{DGH15}. 
Let $\underline{C}$ be a row contraction and $\underline{E}$ be a minimal lifting of $\underline{C}$ by $\underline{A}$ on $\mathcal{H}_E=\mathcal{H}_C\oplus\mathcal{H}_A$. 
Then the minimal isometric dilation $\underline{V}^E$ of $E$ is also an isometric dilation of $\underline{C}$, thus $\hat{\clh}_C$ can be embedded in $\hat{\clh}_E$ as a reducing subspace of $\underline{V}^E.$ We denote the orthogonal complement of the image of $\hat{\clh}_C$ in $\hat{\clh}_E$ by $\clk$ and denote the restriction of ${\underline{V}}^E$ on $\clk$ by $\underline{Y}$.  Then there exist a unitary operator $W:\hat{\clh}_{E}\rightarrow\hat{\clh}_C\oplus\clk$ satisfying

\begin{equation} \label{char uni}
W|_{{\clh_C}} = I_{\clh_C}\text{ and }W{V_i}^{E} = (V^{C}\oplus Y)_iW\; \text{for all}\; i = 1,2,..,d. 
\end{equation}
S. Dey and R. Gohm \cite{DEY} used the operator $W$ to define a characteristic function for minimal contractive lifting of a row contraction as follows.
\begin{definition}\label{def 2.4}
	The characteristic function of the minimal contractive lifting $\underline{E}$ of $\underline{C}$ is the  operator $M_{C,E}:\Gamma\otimes\cld_E\rightarrow\Gamma\otimes\cld_{C}$ defined by
	\[
	M_{C,E} = P_{\Gamma\otimes\cld_C}W|_{\Gamma\otimes\cld_E},
	\] 
	where $P_{\Gamma\otimes\cld_C}$ is the orthogonal projection onto $\Gamma\otimes\cld_C$.
\end{definition}
 The characteristic function $M_{C,E}$ is a multi-analytic operator (cf. \cite{POP}), i.e., $M_{C,E} (L_i \otimes I_{\cld_E})= (L_i \otimes I_{\cld_C}) M_{C,E}$ for $i=1,\ldots,d$ and therefore it is determined by its symbol $\theta_{C,E}:\cld_E\rightarrow\Gamma\otimes \cld_C$, where $\theta_{C,E}:=W|_{e_0\otimes \cld_E}$. We refer to \cite{DEY,DGH15} for detailed discussion about characteristic function of liftings. We frequently use the following representation of $\theta_{C,E}$, which is given in section $4$ of the article \cite{DEY}.\\
 For $ h\in \mathcal{H}_C$
\begin{align}\label{eqn dey characteristic funtion HC}
\theta_{C,E}(D_E)_ih=e_0\otimes[(D_C)_ih-\gamma D_{*,A}B_ih]-\underset{|\alpha|\geq1}{\sum}e_{\alpha}\otimes\gamma D_{*,A}A_{\alpha}^*B_ih,
\end{align}
and for $h\in \mathcal{H}_A$
\begin{align}\label{eqn dey characteristic function HA}
	\theta_{C,E}(D_E)_ih=&-e_0\otimes\gamma \underline{A}(D_A)_ih+\underset{j=1}{\overset{d}{\sum}}e_j\otimes\underset{\alpha}{\sum}e_{\alpha}\otimes\gamma D_{*,A}A_{\alpha}^*P_jD_A(D_A)_ih,
\end{align}
where $\gamma:\cld_{*,A}\rightarrow \cld_C$ is the contraction from Proposition \ref{Prop lifting}.

The unique formal Fourier expansion of $M_{C,E}$ is given by,
\begin{align*}
M_{C,E}&\sim\sum_{\alpha\in\tilde{\Lambda}}R_{\alpha}\otimes (\theta_{C,E})_{(\alpha)}
\end{align*}
where $R_i$ is the right creation operator defined on the Fock space and the ``coefficients" $(\theta_{C,E})_{(\alpha)}  \in \mathcal{B}(\cld_E,\cld_C)$ are given by
$$\langle (\theta_{C,E})_{(\alpha)}x,y\rangle = \langle\theta_{C,E} x, e_{\tilde{\alpha}}y\rangle$$ for $x\in \cld_E$ and for $y\in\cld_C$. The symbol $\tilde{\alpha}$ is the reverse of $\alpha,$ i.e., if $\alpha = \alpha_1\alpha_2\ldots\alpha_n$ then $\tilde{\alpha} = \alpha_n\alpha_{(n-1)}\ldots\alpha_1.$ \\
Let $x=D_Eh$ for $h\in \clh_C$. Then
\begin{align}\label{eqn of fourier coefficient for HC}
	(\theta_{C,E})_{(\alpha)}x =\begin{cases}
	D_Ch-\gamma D_{*A}\underline{B}h&\text{if } \alpha = 0\\
	-\gamma D_{*A}A_{\tilde{\alpha}}^*\underline{B}h&\text{if }\alpha\neq 0
	\end{cases}
\end{align}

Let $x=D_Eh$ for $h\in\clh_A$. Then $\theta_{C,E}x=(I\otimes\gamma)\theta_{A}x$. Hence,
\begin{align}\label{eqn of fourier coefficient for HA}
(\theta_{C,E})_{(\alpha)}x= \gamma{\theta_{A}}_{(\alpha)}x
\end{align}
By von Neumann inequality \cite{GP1} for non commutative operators, we have
\begin{align*}
	{M_{C,E}}&=\text{SOT}-\underset{r\rightarrow 1}{\lim}M_{C,E}(rR)\text{ for }r\in[0,1)\\
	M_{C,E}&(rR)=\sum_{\alpha\in\tilde{\Lambda}}r^{|\alpha|}R_{\alpha}\otimes\theta_{\alpha}
\end{align*}
By Equations (\ref{eqn of fourier coefficient for HC}) and (\ref{eqn of fourier coefficient for HA}), we obtain

\begin{align}
	{M_{C,E}(rR)(I_{\Gamma}\otimes D_E)|_{\Gamma\otimes\mathcal{H}_C}}=&I_{\Gamma}\otimes D_C-(I_{\Gamma}\otimes\gamma D_{*,A})\left(I_{\Gamma\otimes\mathcal{H}_E}-r\underset{i=1}{\overset{d}{\sum}}R_i\otimes A_i^*\right)^{-1}(I_{\Gamma}\otimes D_{*,A}\gamma^*
	 D_C)\\
	 {M_{C,E}(rR)(I_{\Gamma}\otimes D_E)|_{\Gamma\otimes\mathcal{H}_A}}=&-I_{\Gamma}\otimes\gamma \underline{A}D_A+(I_{\Gamma}\otimes \gamma D_{*,A})\left(I_{\Gamma\otimes\mathcal{H}_E}-r\underset{i=1}{\overset{d}{\sum}}R_i\otimes A_i^*\right)^{-1}
	 r(\underline{R}\otimes I_{\mathcal{H}_E})(I_{\Gamma}\otimes D_A^2)
\end{align}
For a row contraction $\underline{T} = (T_1,T_2,\ldots,T_n)$, we denote $(I_{\Gamma}\otimes T_1,\ldots,I_{\Gamma}\otimes T_n)$   by $\underline{T}_{\Gamma}.$ Also $I_{\Gamma}\otimes\gamma$  and $(R_1\otimes I_{H},\ldots,R_n\otimes I_{H})$ are denoted by $\gamma_{\Gamma}$ and $\underline{R}_{H}.$ Then
\begin{align}
\label{eqn characteristic for HC}
M_{C,E}(rR)D_{E_{\Gamma}}|_{\Gamma\otimes\mathcal{H}_C}=& D_{C_{\Gamma}}-\gamma_{\Gamma} D_{*,A_{\Gamma}}\left( I-r \underline{R}_H\underline{A}^*_{\Gamma}\right)^{-1}D_{*,A_{\Gamma}}\gamma^*_{\Gamma} D_{C_{\Gamma}},\\
\label{eqn characteristic for HA}	
{M_{C,E}(rR)D_{E_{\Gamma}}|_{\Gamma\otimes\mathcal{H}_A}}=&\left(-\gamma_{\Gamma} \underline{A}_{\Gamma}+\gamma_{\Gamma} D_{*,A_{\Gamma}}(I-r \underline{R}_H \underline{A}_{\Gamma}^*)^{-1}r \underline{R}_H D_{A_{\Gamma}}\right) D_{A_{\Gamma}}.
\end{align}

For a row contraction $\underline{A}$ on a Hilbert space $\mathcal{H}$, we denote the characteristic function of $\underline{A}$ (cf. \cite{GP2}) by {$M_{A}$. In \cite{GP1}, it is shown that  $M_{A}=\text{SOT}-\underset{r\rightarrow 1}{\lim}M_{A}(rR)$ where
	\[
	M_{A}(rR)=- \underline{A}_{\Gamma}+ D_{*,A_{\Gamma}}(I-r \underline{R}_H \underline{A}_{\Gamma}^*)^{-1}r \underline{R}_H D_{A_{\Gamma}},
	\]
	and its symbol  $\theta_{A}:\mathcal{D}_A\rightarrow\Gamma\otimes\mathcal{D}_{*,A}$ is defined by $\theta_{A}=M_A|_{e_0\otimes\mathcal{D}_A}.$ } Forthcoming result is used very frequently in the subsequent sections.
\begin{lemma}\cite[Lemma 2.1]{MAJI}\label{Lemma MAJI}
	For a row contraction $\underline{A}$ on $\mathcal{H}$ and for $r\in[0,1)$, we have 
	$$ D_{*,A_{\Gamma}}(I-r \underline{R}_H \underline{A}_{\Gamma}^*)^{-1} D_{*,A_{\Gamma}}=I_{ \cld_{*,A_{\Gamma}}}+M_{A}(rR) \underline{A}_{\Gamma}^*.$$
\end{lemma}

\section{Factorization of an iterated lifting}
In this section, our aim is to factorize the characteristic function of an iterative lifting. If we start with a row contraction $\underline{C}$ and  consider a contractive lifting $\underline{E}$ of $\underline{C}$ by $\underline{A}$ and $\underline{E}'$ is a contractive lifting of $\underline{E}$ by $\underline{A}'$, then we want to factorize $M_{C,E'}$ in terms of $M_{A}$ and $M_{A'}$.

We start with a result, which gives a relation between the defect spaces of contractions appearing in a lifting.
\begin{lemma}\label{lemma isomtery}
	Let $\underline{C}=[C_1,\ldots,C_d]$ and $\underline{A}=[A_1,\ldots,A_d]$ be two row contractions on Hilbert spaces $\mathcal{H}_C$ and $\mathcal{H}_A$, respectively. Suppose $\delta\in\mathcal{B}(\cld_{*,A},\cld_C)$ is a contraction and $\underline{E}$ is a contractive lifting of $\underline{C}$ by $\underline{A}$ on $\mathcal{H}_E=\mathcal{H}_C\oplus \mathcal{H}_A$ such that
	\[E_i=\begin{bmatrix}
		C_i&0\\
		 (D_{*,A}\delta^* D_C)_i&A_i
	\end{bmatrix}\text{ for }i=1,2,\ldots d.\]
Then there exist {a unitary} $\sigma_E:\cld_E\rightarrow \cld_{*,\delta}\oplus \cld_A$ such that 
\[\sigma_E D_E=\begin{bmatrix}
	 D_{*,\delta} D_C&0\\
	-A^*\delta^* D_C& D_A
\end{bmatrix}.\]
\end{lemma}	
\begin{proof}
	Let $h=h_C\oplus h_A\in \mathcal{H}_C\oplus \mathcal{H}_A$. Then 
	\begin{align*}
		\| D_Eh\|^2=&\|h\|^2-\|\underline{E}h\|^2\\
		=&\|h_C\|^2+\|h_A\|^2-\|\underline{C}h_C\|^2-\| D_{*,A}\delta^* D_Ch_C+\underline{A}h_A\|^2\\
		=&\| D_Ch_C\|^2+\|h_A\|^2-\| D_{*,A}\delta^* D_Ch_C\|^2-\|\underline{A}h_A\|^2-\langle D_{*,A}\delta^* D_Ch_C,\underline{A}h_A\rangle\\
		&-\langle \underline{A}h_A, D_{*,A}\delta^* D_Ch_C\rangle\\
		=&\| D_Ch_C\|^2+\| D_Ah_A\|^2-\|\delta^* D_Ch_C\|^2+\|\underline{A}^*\delta^* D_Ch_C\|^2-\langle\delta^* D_Ch_C, D_{*,A}\underline{A}h_A\rangle\\
		&-\langle D_{*,A}\underline{A}h_A,\delta^* D_Ch_C\rangle\\
		=&\| D_{*,\delta} D_Ch_C\|^2+\| D_Ah_A\|^2+\|\underline{A}^*\delta^* D_Ch_C\|^2-\langle\delta^* D_Ch_C,\underline{A} D_{A}h_A\rangle-\langle \underline{A} D_{A}h_A,\delta^* D_Ch_C\rangle\\
		=&\| D_{*,\delta} D_Ch_C\|^2+\| D_Ah_A\|^2+\|\underline{A}^*\delta^* D_Ch_C\|^2-\langle \underline{A}^*\delta^* D_Ch_C, D_{A}h_A\rangle-\langle  D_{A}h_A,\underline{A}^*\delta^* D_Ch_C\rangle\\
		=&\| D_{*,\delta} D_Ch_C\|+\| D_Ah_A-\underline{A}^*\delta^* D_Ch_C\|^2.
	\end{align*}
Thus \[\sigma_E D_E=\begin{bmatrix}
	 D_{*,\delta} D_C&0\\
	-\underline{A}^*\delta^* D_C& D_A
\end{bmatrix}\] is an isometry from $\cld_E$ onto $\cld_{*,\delta}\oplus \cld_A$.
\end{proof} 
Similar to the previous result, the following lemma gives another relation between the defect spaces  of contractions appearing in a lifting.
\begin{lemma}
	\label{lemma adjoint isometry}
	Let $\underline{C},\,\underline{A},\,\delta$ and $\underline{E}$ be as defined in Lemma \ref{lemma isomtery}. Then there exist {a unitary} $\sigma'_E:\cld_{*,E}\rightarrow \cld_{*,C}\oplus \cld_{\delta}$ such that 
	\begin{align}\label{adjoint isometry}
			\sigma '_E D_{*,E}=\begin{bmatrix}
			D_{*,C}&-\underline{C}\delta D_{*,A}\\
			0& D_{\delta} D_{*,A}
		\end{bmatrix}.
	\end{align}
\end{lemma}
\begin{proof}
	Let $h=h_C\oplus h_A\in \mathcal{H}_C\oplus \mathcal{H}_A$. Then 
\begin{align*}
	\| D_{*,E}h\|^2=&\|h\|^2-\|\underline{E}^*h\|^2\\
	=&\|h_C\|^2+\|h_A\|^2-\|\underline{C}^*h_C+ D_C\delta D_{*,A}h_A\|^2-\|\underline{A}^*h_A\|^2\\
	=&\|h_C\|^2+\| D_{*,A}h_A\|^2-\| D_C\delta D_{*,A}h_A\|^2-\|\underline{C}^*h_C\|^2-\langle D_C\delta D_{*,A}h_A,\underline{C}^*h_C\rangle\\
	&-\langle \underline{C}^*h_C, D_C\delta D_{*,A}h_A\rangle\\
	=&\| D_{*,C}h_C\|^2+\| D_{*,A}h_A\|^2-\|\delta D_{*,A}h_A\|^2+\|\underline{C}\delta D_{*,A}h_A\|^2-\langle\delta D_{*,A}h_A, D_C\underline{C}^*h_C\rangle\\
	&-\langle D_C\underline{C}^*h_C,\delta D_{*,A}h_A\rangle\\
=&\| D_{*,C}h_C\|^2+\| D_{\delta} D_{*,A}h_A\|^2+\|\underline{C}\delta D_{*,A}h_A\|^2-\langle\delta D_{*,A}h_A,\underline{C}^* D_{*,C}h_C\rangle\\
&-\langle \underline{C}^* D_{*,C}h_C,\delta D_{*,A}h_A\rangle\\
=&\| D_{*,C}h_C\|^2+\| D_{\delta} D_{*,A}h_A\|^2+\|\underline{C}\delta D_{*,A}h_A\|^2-\langle \underline{C}\delta D_{*,A}h_A, D_{*,C}h_C\rangle\\
&-\langle  D_{*,C}h_C,\underline{C}\delta D_{*,A}h_A\rangle\\
=&\| D_{*,C}h_C-\underline{C}\delta D_{*,A}h_A\|^2+\| D_{\delta} D_{*,A}h_A\|^2.
\end{align*}
This gives the required isometry {from $\cld_{*,E}$ onto $\cld_{*,C}\oplus \cld_{\delta}$} defined in Equation (\ref{adjoint isometry}).
\end{proof}
Let $\underline{C}$ be a row contraction on $\mathcal{H}_C$, $\underline{E}$ be a minimal lifting of $\underline{C}$ by $\underline{A}$ on $\mathcal{H}_E=\mathcal{H}_C\oplus \mathcal{H}_A$, and $\underline{E}'$ be a minimal lifting of $\underline{E}$ by $\underline{A}'$ on $\mathcal{H}_{E'}=\mathcal{H}_E\oplus \mathcal{H}_{A'}$. Then there are $B_i \in B(\mathcal{H}_C, \mathcal{H}_A)$ and $B_i' \in B(\mathcal{H}_E, \mathcal{H}_{A'})$ such that
\begin{align}\label{equation steplifting}
	E_i=\begin{bmatrix}
		C_i&0\\
		B_i&A_i
	\end{bmatrix},\,E_i'=\begin{bmatrix}
	E_i&0\\
	B_i'&A_i'
\end{bmatrix}\text{ for }i=1,2,\ldots d.
\end{align} 
Also, for some operators $(B_1')_i$ and $(B_2')_i$ 
\[E_i'=\begin{bmatrix}
	C_i&0&0\\
	B_i&A_i&0\\
	(B_1')_i&(B_2')_i&A_i'
\end{bmatrix}\text{ where }B_i'=\begin{bmatrix}
(B_1')_i&(B_2')_i
\end{bmatrix}.\]
We set
\begin{equation} \label{contractions}
\hat{A}_i=\begin{bmatrix}
	A_i&0\\
	(B_2')_i&A_i'
\end{bmatrix}\text{ and }\hat{B}_i=\begin{bmatrix}
B_i\\
(B_1')_i
\end{bmatrix}.
\end{equation}
Then
\[ E_i'=\begin{bmatrix}
	C_i&0\\
	\hat{B}_i&\hat{A}_i
\end{bmatrix}. \]

\begin{remark} \label{delta}
	By Proposition \ref{Prop lifting}, there exist contractions $\gamma:\cld_{*,A}\rightarrow \cld_C$, $\gamma':\cld_{*,A'}\rightarrow \cld_E$, $\delta:\cld_{*, A'}\rightarrow \cld_A$ and $\hat{\gamma}:\cld_{*,\hat{A}}\rightarrow \cld_C,$ 
	 satisfying
	\begin{align*}
		\underline{B}= D_{*,A}\gamma^* D_C,\,\underline{B}'= D_{*,A'}\gamma'^* D_E,\,\underline{B_2}'= D_{*,A'}\delta^* D_A,\text{ and }\underline{\hat{B}}= D_{*,\hat{A}}\hat{\gamma}^* D_C.
	\end{align*}
\end{remark}
Now, for the characteristic function $M_{C,E'}$ we prove a factorization for $M_{C,E'}D_{E'_{\Gamma}}$, where the domain is restricted to $\Gamma\otimes \mathcal{H}_{\hat{A}}.$ 
\begin{proposition}\label{prop second component}
Let $\underline{C}$ be a row contraction on $\mathcal{H}_C$.	Let $\underline{E},\underline{E'},\underline{\hat{A}}, \hat{\gamma}$ and $\delta$ be as defined in Equations (\ref{equation steplifting}) and (\ref{contractions}), and Remark \ref{delta}. Then 
	\begin{equation*}
		{M_{C,E'}D_{E'_{\Gamma}}|_{\Gamma\otimes\mathcal{H}_{\hat{A}}}}=(I_{\Gamma}\otimes\hat{\gamma})\left(I_{\Gamma}\otimes(\sigma'_{\hat{A}})^{-1}\right)\begin{bmatrix}
			{M_{A}}&0\\
			0&I_{\Gamma\otimes{\cld}_{*,A'}}
		\end{bmatrix}\Bigg(I_{\Gamma}\otimes\begin{bmatrix}
		 D_{*,\delta}&\delta\\
		-\delta^*& D_{\delta}
	\end{bmatrix}\Bigg)\begin{bmatrix}
	I_{\Gamma\otimes{\cld}_A}&0\\
	0&{M_{A'}}
\end{bmatrix}
(I_{\Gamma}\otimes\sigma_{\hat{A}}) (I_{\Gamma}\otimes D_{\hat{A}})
\end{equation*}
where $\sigma_{\hat{A}}$ and $\sigma'_{\hat{A}}$  are the unitaries defined in Lemmas \ref{lemma isomtery} and \ref{lemma adjoint isometry} for the lifting $\underline{\hat{A}}$ of $\underline{A}.$ 
\end{proposition}
\begin{proof}
By Equation (\ref{eqn characteristic for HA}), we have the following for $h\in \Gamma\otimes\mathcal{H}_{\hat{A}}$ and for $0<r<1$,
	\begin{align*}
		{M_{C,E'}(rR)D_{E'_{\Gamma}}}h=&(I_{\Gamma}\otimes\hat{\gamma})\left(-\underline{\hat{A}}_{\Gamma} D_{\hat{A}_{\Gamma}}+  D_{*,\hat{A}_{\Gamma}}\left(I-r  \underline{R}_H\underline{\hat{A}}_{\Gamma}^*\right)^{-1}r \underline{R}_H  D_{\hat{A}_{\Gamma}}^2\right)h\\
		=&(I_{\Gamma}\otimes\hat{\gamma})\left(- D_{*,\hat{A}_{\Gamma}}\underline{\hat{A}}_{\Gamma}+D_{*,\hat{A}_{\Gamma}}\left(I-r \underline{R}_H\underline{\hat{A}}_{\Gamma}^*\right)^{-1}r \underline{R}_H  D_{\hat{A}_{\Gamma}}^2\right)h.
	\end{align*} 
By Lemma \ref{lemma adjoint isometry}, we have \begin{align}\label{eqn first component}
	\sigma'_{\hat{A}} D_{*,\hat{A}}\hat{A}=&\begin{bmatrix}
		 D_{*,A}&-\underline{A}\delta D_{*,A'}\\
		0& D_{\delta} D_{*,A'}
	\end{bmatrix}\begin{bmatrix}
	\underline{A}&0\\
	\underline{B_2}'&\underline{A}' \nonumber
\end{bmatrix}\\
=&\begin{bmatrix}
	 D_{*,A}\underline{A}-\underline{A}\delta D_{*,A'}\underline{B_2}'&-\underline{A}\delta D_{*,A'}\underline{A}'\\
	 D_{{\delta}} D_{*,{A'}}\underline{B}_2'& D_{\delta} D_{*,{A'}}\underline{A}'
\end{bmatrix}. 
\end{align}
Note that 
\begin{equation*}
	\left(I-r \underline{R}_H\underline{\hat{A}}^*_{\Gamma}\right)^{-1}=\begin{bmatrix}
		(I-r \underline{R}_H\underline{A}^*_{\Gamma})^{-1}&(I-r \underline{R}_H\underline{A}^*_{\Gamma})^{-1}(r \underline{R}_H\underline{B_2}'^*_{\Gamma})(I-r \underline{R}_H\underline{A'}^*_{\Gamma})^{-1}\\
		0&(I-r \underline{R}_H\underline{A'}^*_{\Gamma})^{-1}
	\end{bmatrix}.
\end{equation*}
We estimate $\alpha=:(I_{\Gamma}\otimes\sigma'_{\hat{A}}) D_{*,\hat{A}_{\Gamma}}(I-r \underline{R}_H\underline{\hat{A}}_{\Gamma}^*)^{-1}$,  as follows 
\begin{align}\label{product operator}
	\alpha=&\begin{bmatrix}
		 D_{*,A_{\Gamma}}&- \underline{A}_{\Gamma}{\underline{\delta}_{\Gamma}} D_{*,A'_{\Gamma}}\\
		0& D_{\delta_{\Gamma}} D_{*,A'_{\Gamma}}
	\end{bmatrix}
\begin{bmatrix}
	(I-r \underline{R}_H \underline{A}_{\Gamma}^*)^{-1}&(I-r \underline{R}_H \underline{A}_{\Gamma}^*)^{-1}r \underline{R}_H\underline{B_2'^*}_{\Gamma}(I-r \underline{R}_H\underline{A'}_{\Gamma}^*)^{-1}\\
	0&(I-r \underline{R}_H\underline{A'}_{\Gamma}^*)^{-1}
\end{bmatrix}\\
=&\begin{bmatrix}
	 D_{*,A_{\Gamma}}(I-r \underline{R}_H \underline{A}_{\Gamma}^*)^{-1}& D_{*,A_{\Gamma}}(I-r \underline{R}_H \underline{A}_{\Gamma}^*)^{-1}r \underline{R}_H\underline{B_2'^*}_{\Gamma}(I-r \underline{R}_H\underline{A'}_{\Gamma}^*)^{-1}\\
	&- \underline{A}_{\Gamma}\delta_{\Gamma} D_{*,A'_{\Gamma}}(I-r \underline{R}_H\underline{A'}_{\Gamma}^*)^{-1}\\
	0& D_{\delta_{\Gamma}} D_{*,A'_{\Gamma}}(I-r \underline{R}_H\underline{A'}_{\Gamma}^*)^{-1}
\end{bmatrix}. \nonumber
\end{align}
By simple matrix multiplication, we get 
\begin{equation}\label{defect oper}
	 D_{\hat{A}}^2=\begin{bmatrix}
		 D_A^2-\underline{B_2'}^*\underline{B_2'}&-\underline{B_2'}^*\underline{A}'\\
		-\underline{A}'^*\underline{B_2'}& D_{A'}^2
	\end{bmatrix}.
\end{equation}
We denote 
\begin{equation}\label{eqn second component}
	(I_{\Gamma}\otimes\sigma'_{\hat{A}}) D_{*,\hat{A}_{\Gamma}}(I-r \underline{R}_H\underline{\hat{A}}_{\Gamma}^*)^{-1}r \underline{R}_H D_{\hat{A}_{\Gamma}}^2:=\begin{bmatrix}
		K&L\\
		M&N
	\end{bmatrix}.
\end{equation}
By Equations (\ref{product operator}) and (\ref{defect oper}), we compute $K,\,L,\,M$ and $N$. Using the fact that $\underline{B_2}'^*= D_A\delta D_{*,A'}$ (by Proposition \ref{Prop lifting}), we have
 \begin{align*}
	K=& D_{*,A_{\Gamma}}(I-r \underline{R}_H \underline{A}_{\Gamma}^*)^{-1}r \underline{R}_H D_{A_{\Gamma}}^2- D_{*,A_{\Gamma}}(I-r \underline{R}_H \underline{A}_{\Gamma}^*)^{-1}r \underline{R}_H\underline{B_2'}_{\Gamma}^*\underline{B_2'}_{\Gamma}- D_{*,A_{\Gamma}}(I-r \underline{R}_H \underline{A}_{\Gamma}^*)^{-1}r \underline{R}_H\\
	&\underline{B_2'^*}_{\Gamma}(I-r \underline{R}_H\underline{A'}_{\Gamma}^*)^{-1}r \underline{R}_H\underline{A'}_{\Gamma}^*\underline{B_2'}_{\Gamma}
	+ \underline{A}_{\Gamma}\delta_{\Gamma} D_{*,A'_{\Gamma}}(I-r \underline{R}_H\underline{A'}_{\Gamma}^*)^{-1}r \underline{R}_H\underline{A'}_{\Gamma}^*\underline{B_2'}_{\Gamma}\\
	=& D_{*,A_{\Gamma}}(I-r \underline{R}_H \underline{A}_{\Gamma}^*)^{-1}r \underline{R}_H D_{A_{\Gamma}}^2- D_{*,A_{\Gamma}}(I-r \underline{R}_H \underline{A}_{\Gamma}^*)^{-1}r \underline{R}_H\underline{B_2'}_{\Gamma}^*\underline{B_2'}_{\Gamma}-[- \underline{A}_{\Gamma}+ D_{*,A_{\Gamma}}(I-r \underline{R}_H \underline{A}_{\Gamma}^*)^{-1}\\
	&r \underline{R}_H D_{A_{\Gamma}}]\delta_{\Gamma} D_{*,A'_{\Gamma}}(I-r \underline{R}_H\underline{A'}_{\Gamma}^*)^{-1}r \underline{R}_H\underline{A'}_{\Gamma}^*\underline{B_2'}_{\Gamma}\\
	=& D_{*,A_{\Gamma}}(I-r \underline{R}_H \underline{A}_{\Gamma}^*)^{-1}r \underline{R}_H D_{A_{\Gamma}}^2- D_{*,A_{\Gamma}}(I-r \underline{R}_H \underline{A}_{\Gamma}^*)^{-1}r \underline{R}_H\underline{B_2'}_{\Gamma}^*\underline{B_2'}_{\Gamma}-{M_{A}}(rR)\delta_{\Gamma} D_{*,A'_{\Gamma}}(I-r \underline{R}_H\underline{A'}_{\Gamma}^*)^{-1}\\
	&r \underline{R}_H\underline{A'}_{\Gamma}^*\underline{B_2'}_{\Gamma},\\
	L=&- D_{*,A_{\Gamma}}(I-r \underline{R}_H \underline{A}_{\Gamma}^*)^{-1}r \underline{R}_H\underline{B_2'^*}_{\Gamma}\underline{A'}_{\Gamma}+ D_{*,A_{\Gamma}}(I-r \underline{R}_H \underline{A}_{\Gamma}^*)^{-1}r \underline{R}_H\underline{B_2'}_{\Gamma}^*(I-r \underline{R}_H\underline{A'}_{\Gamma}^*)^{-1}r \underline{R}_H D_{A'_{\Gamma}}^2\\
	&- \underline{A}_{\Gamma}\delta_{\Gamma} D_{*,A'_{\Gamma}}(I-r \underline{R}_H\underline{A'}_{\Gamma}^*)^{-1}r \underline{R}_H D_{A'_{\Gamma}}^2\\
	=&- D_{*,A_{\Gamma}}(I-r \underline{R}_H \underline{A}_{\Gamma}^*)^{-1}r \underline{R}_H\underline{B_2'}_{\Gamma}^*\underline{A'}_{\Gamma}+[- \underline{A}_{\Gamma}+ D_{*,A_{\Gamma}}(I-r \underline{R}_H \underline{A}_{\Gamma}^*)^{-1}r \underline{R}_H D_{A_{\Gamma}}]\delta_{\Gamma} D_{*,A'_{\Gamma}}\\
	&(I-r \underline{R}_H\underline{A'}_{\Gamma}^*)^{-1}
	r \underline{R}_H D_{A'_{\Gamma}}^2\\
	=&- D_{*,A_{\Gamma}}(I-r \underline{R}_H \underline{A}_{\Gamma}^*)^{-1}r \underline{R}_H\underline{B_2'}_{\Gamma}^*\underline{A'}_{\Gamma}+{M_{A}}(rR)\delta_{\Gamma} D_{*,A'_{\Gamma}}(I-r \underline{R}_H\underline{A'}_{\Gamma}^*)^{-1}r \underline{R}_H D_{A'_{\Gamma}}^2,
\end{align*}
and
\begin{align*}
	M=&- D_{\delta_{\Gamma}} D_{*,A'_{\Gamma}}(I-r \underline{R}_H\underline{A'}_{\Gamma}^*)^{-1}r \underline{R}_H\underline{A'}_{\Gamma}^*\underline{B_2'}_{\Gamma},\\
	N=& D_{\delta_{\Gamma}} D_{*,A'_{\Gamma}}(I-r \underline{R}_H\underline{A'}_{\Gamma}^*)^{-1}r \underline{R}_H D_{A'_{\Gamma}}^2.
\end{align*}
Let
$$-\left(I_{\Gamma}\otimes\sigma'_{\hat{A}}\right) D_{\hat{A}_{\Gamma}^*}\underline{\hat{A}}_{\Gamma}+	(I_{\Gamma}\otimes\sigma'_{\hat{A}}) D_{*,\hat{A}_{\Gamma}}(I-r \underline{R}_H\underline{\hat{A}}_{\Gamma}^*)^{-1}r \underline{R}_H D_{\hat{A}_{\Gamma}}^2:=\begin{bmatrix}
	S&X\\
	Y&Z
\end{bmatrix}.$$
By Equations (\ref{eqn first component}) and (\ref{eqn second component}), we have
\begin{equation*}
	\begin{bmatrix}
		S&X\\
		Y&Z
	\end{bmatrix}=-\begin{bmatrix}
	 D_{*,A_{\Gamma}} \underline{A}_{\Gamma}- \underline{A}_{\Gamma}\delta_{\Gamma} D_{*,A'_{\Gamma}}\underline{B_2'}_{\Gamma}&- \underline{A}_{\Gamma}\delta_{\Gamma} D_{*,A'_{\Gamma}}\underline{A'}_{\Gamma}\\
	 D_{\delta_{\Gamma}} D_{*,A'_{\Gamma}}\underline{B_2'}_{\Gamma}& D_{\delta_{\Gamma}} D_{*,A'_{\Gamma}}\underline{A'}_{\Gamma}
\end{bmatrix}+\begin{bmatrix}
	K&L\\
	M&N
\end{bmatrix}.
\end{equation*}
Again using $\underline{B_2}'= D_{*,A'}\delta^* D_A$ and Lemma \ref{Lemma MAJI},
\begin{align*}
	S=&- D_{*,A_{\Gamma}} \underline{A}_{\Gamma}+ \underline{A}_{\Gamma}\delta_{\Gamma} D_{*,A'_{\Gamma}}\underline{B_2'}_{\Gamma}+ D_{*,A_{\Gamma}}(I-r \underline{R}_H \underline{A}_{\Gamma}^*)^{-1}r \underline{R}_H D_{A_{\Gamma}}^2- D_{*,A_{\Gamma}}(I-r \underline{R}_H \underline{A}_{\Gamma}^*)^{-1}r \underline{R}_H\underline{B_2'}_{\Gamma}^*\underline{B_2'}_{\Gamma}\\
	&-{M_{A}}(rR)\delta_{\Gamma} D_{*,A'_{\Gamma}}(I-r \underline{R}_H\underline{A'}_{\Gamma}^*)^{-1}r \underline{R}_H\underline{A'}_{\Gamma}^*\underline{B_2'}_{\Gamma}\\
	=&\left(- \underline{A}_{\Gamma} D_{A_{\Gamma}}+ D_{*,A_{\Gamma}}(I-r \underline{R}_H \underline{A}_{\Gamma}^*)^{-1}r \underline{R}_H D_{A_{\Gamma}}^2\right)-\left(- \underline{A}_{\Gamma}+ D_{*,A_{\Gamma}}(I-r \underline{R}_H \underline{A}_{\Gamma}^*)^{-1}r \underline{R}_H D_{A_{\Gamma}}\right)\delta_{\Gamma} D_{*,A'_{\Gamma}}\underline{B_2'}_{\Gamma}\\
	&-{M_{A}}(rR)\delta_{\Gamma} D_{*,A'_{\Gamma}}(I-r \underline{R}_H\underline{A'}_{\Gamma}^*)^{-1}r \underline{R}_H\underline{A'}_{\Gamma}^*\underline{B_2'}_{\Gamma}\\
	=&{M_{A}}(rR) D_{A_{\Gamma}}-{M_{A}}(rR)\delta_{\Gamma} D_{*,A'_{\Gamma}}\underline{B_2'}_{\Gamma}-{M_{A}}(rR)\delta_{\Gamma} D_{*,A'_{\Gamma}}(I-r \underline{R}_H\underline{A'}_{\Gamma}^*)^{-1}r \underline{R}_H\underline{A'}_{\Gamma}^*\underline{B_2'}_{\Gamma}\\
	=&{M_{A}}(rR) D_{A_{\Gamma}}-{M_{A}}(rR)\delta_{\Gamma} D_{*,A'_{\Gamma}}(I-r \underline{R}_H\underline{A'}_{\Gamma}^*)^{-1}[I-r \underline{R}_H\underline{A'}_{\Gamma}^*+r \underline{R}_H\underline{A'}_{\Gamma}^*]\underline{B_2'}_{\Gamma}\\
	=&{M_{A}}(rR) D_{A_{\Gamma}}-{M_{A}}(rR)\delta_{\Gamma} D_{*,A'_{\Gamma}}(I-r \underline{R}_H\underline{A'}_{\Gamma}^*)^{-1} D_{*,A'_{\Gamma}}\delta_{\Gamma}^* D_{A_{\Gamma}}\\
	=&{M_{A}}(rR) D_{A_{\Gamma}}-{M_{A}}(rR)\delta_{\Gamma}[I+{M_{A'}}(rR)\underline{A'}_{\Gamma}^*]\delta_{\Gamma}^* D_{A_{\Gamma}}\\
	=&{M_{A}}(rR) D_{A_{\Gamma}}-{M_{A}}(rR)\delta_{\Gamma}\delta_{\Gamma}^* D_{A_{\Gamma}}-{M_{A}}(rR)\delta_{\Gamma}{M_{A'}}(rR)\underline{A'}_{\Gamma}^*\delta_{\Gamma}^* D_{A_{\Gamma}}\\
	=&{M_{A}}(rR) D_{*,\delta_{\Gamma}}^2 D_{A_{\Gamma}}-{M_{A}}(rR)\delta_{\Gamma}{M_{A'}}(rR)\underline{A'}_{\Gamma}^*\delta_{\Gamma}^* D_{A_{\Gamma}},
\end{align*}
\begin{align*}
	X=& \underline{A}_{\Gamma}\delta_{\Gamma} D_{*,A'_{\Gamma}}\underline{A'}_{\Gamma}- D_{*,A_{\Gamma}}(I-r \underline{R}_H \underline{A}_{\Gamma}^*)^{-1}r \underline{R}_H\underline{B_2'}_{\Gamma}^*\underline{A'}_{\Gamma}+{M_{A}}(rR)\delta_{\Gamma} D_{*,A'_{\Gamma}}(I-r \underline{R}_H\underline{A'}_{\Gamma}^*)^{-1}r \underline{R}_H D_{A'_{\Gamma}}^2\\
	=&-[- \underline{A}_{\Gamma}+ D_{*,A_{\Gamma}}(I-r \underline{R}_H \underline{A}_{\Gamma}^*)^{-1}r \underline{R}_H D_{A_{\Gamma}}]\delta_{\Gamma} D_{*,A'_{\Gamma}}\underline{A'}_{\Gamma}+{M_{A}}(rR)\delta_{\Gamma} D_{*,A'_{\Gamma}}(I-r \underline{R}_H\underline{A'}_{\Gamma}^*)^{-1}r \underline{R}_H D_{A'_{\Gamma}}^2\\
	=&-{M_{A}}(rR)\delta_{\Gamma} D_{*,A'_{\Gamma}}\underline{A'}_{\Gamma}+{M_{A}}(rR)\delta_{\Gamma} D_{*,A'_{\Gamma}}(I-r \underline{R}_H\underline{A'}_{\Gamma}^*)^{-1}r \underline{R}_H D_{A'_{\Gamma}}^2\\
	=&{M_{A}}(rR)\delta_{\Gamma}\left(- D_{*,A'_{\Gamma}}\underline{A'}_{\Gamma}+ D_{*,A'_{\Gamma}}(I-r \underline{R}_H\underline{A'}_{\Gamma}^*)^{-1}r \underline{R}_H D_{A'_{\Gamma}}^2\right)\\
	=&{M_{A}}(rR)\delta_{\Gamma}{M_{A'}}(rR) D_{A'_{\Gamma}},\\
	Y=&- D_{\delta_{\Gamma}} D_{*,A'_{\Gamma}}\underline{B_2'}_{\Gamma}- D_{\delta_{\Gamma}} D_{*,A'_{\Gamma}}(I-r \underline{R}_H\underline{A'}_{\Gamma}^*)^{-1}r \underline{R}_H\underline{A'}_{\Gamma}^*\underline{B_2'}_{\Gamma}\\
	=&- D_{\delta_{\Gamma}} D_{*,A'_{\Gamma}}(I-r \underline{R}_H\underline{A'}_{\Gamma}^*)^{-1}[I-r \underline{R}_H\underline{A'}_{\Gamma}^*+r \underline{R}_H\underline{A'}_{\Gamma}^*]\underline{B_2'}_{\Gamma}\\
	=&- D_{\delta_{\Gamma}} D_{*,A'_{\Gamma}}(I-r \underline{R}_H\underline{A'}_{\Gamma}^*)^{-1}\underline{B_2'}_{\Gamma}\\
	=&- D_{\delta_{\Gamma}} D_{*,A'_{\Gamma}}(I-r \underline{R}_H\underline{A'}_{\Gamma}^*)^{-1} D_{*,A'_{\Gamma}}\delta_{\Gamma}^* D_{A_{\Gamma}}\\
	=&- D_{\delta_{\Gamma}}[I+{M_{A'}}(rR)\underline{A'}_{\Gamma}^*]\delta_{\Gamma}^* D_{A_{\Gamma}}\\
	=&-\delta_{\Gamma}^* D_{*,\delta_{\Gamma}} D_{A_{\Gamma}}- D_{\delta_{\Gamma}}{M_{A'}}(rR)\underline{A'}_{\Gamma}^*\delta_{\Gamma}^* D_{A_{\Gamma}},\\
\end{align*}
and
\begin{align*}
	Z=&- D_{\delta_{\Gamma}} D_{*,A'_{\Gamma}}\underline{A'}_{\Gamma}+ D_{\delta_{\Gamma}} D_{*,A'_{\Gamma}}(I-r \underline{R}_H\underline{A'}_{\Gamma}^*)^{-1}r \underline{R}_H D_{A'_{\Gamma}}^2\\
	=& D_{\delta_{\Gamma}}\left(- D_{*,A'_{\Gamma}}\underline{A'}_{\Gamma}+ D_{*,A'_{\Gamma}}(I-r \underline{R}_H\underline{A'}^*_{\Gamma})^{-1}r \underline{R}_H D_{A'_{\Gamma}}^2\right)\\
	=& D_{\delta_{\Gamma}}{M_{A'}}(rR) D_{A'_{\Gamma}}
\end{align*}
We know that
$${M_{C,E'}(rR)D_{E'_{\Gamma}}|_{\Gamma\otimes\mathcal{H}_{\hat{A}}}}=(I_{\Gamma}\otimes\hat{\gamma})\left(I_{\Gamma}\otimes(\sigma'_{\hat{A}})^{-1}\right)\begin{bmatrix}
	S&X\\
	Y&Z
\end{bmatrix},$$
where
\begin{align*}
\begin{bmatrix}
	S&X\\
	Y&Z
\end{bmatrix}=&\begin{bmatrix}
{M_{A}}(rR) D_{*,\delta_{\Gamma}}^2 D_{
	A_{\Gamma}}-{M_{A}}(rR)\delta_{\Gamma}M_{A'}(rR)\underline{A'}_{\Gamma}^*\delta_{\Gamma}^* D_{A_{\Gamma}}&{M_{A}}(rR)\delta_{\Gamma}{M_{A'}}(rR) D_{A'_{\Gamma}}\\
-\delta_{\Gamma}^* D_{*,\delta_{\Gamma}} D_{A_{\Gamma}}- D_{\delta_{\Gamma}}{M_{A'}}(rR)\underline{A'}_{\Gamma}^*\delta_{\Gamma}^* D_{A_{\Gamma}}& D_{\delta_{\Gamma}}{M_{A'}}(rR) D_{A'_{\Gamma}}
\end{bmatrix}\\
=&\begin{bmatrix}
	{M_{A}}(rR) D_{*,\delta_{\Gamma}}&{M_{A}}(rR)\delta_{\Gamma}{M_{A'}}(rR)\\
	-\delta_{\Gamma}^*& D_{\delta_{\Gamma}}{M_{A'}}(rR)
\end{bmatrix}\begin{bmatrix}
 D_{*,\delta_{\Gamma}} D_{A_{\Gamma}}&0\\
-\underline{A'}_{\Gamma}^*\delta_{\Gamma}^* D_{A_{\Gamma}}& D_{A'_{\Gamma}}
\end{bmatrix}\\
=&\begin{bmatrix}
	{M_{A}}(rR) D_{*,\delta_{\Gamma}}&{M_{A}}(rR)\delta_{\Gamma}{M_{A'}}(rR)\\
	-\delta_{\Gamma}^*& D_{\delta_{\Gamma}}{M_{A'}}(rR)
\end{bmatrix}(I_{\Gamma}\otimes\sigma_{\hat{A}}) D_{\hat{A}_{\Gamma}}\\
=&\begin{bmatrix}
	{M_{A}}(rR)&0\\
	0& I_{\Gamma\otimes {\cld}_{*,A'}}
\end{bmatrix}\begin{bmatrix}
 D_{*,\delta_{\Gamma}}&\delta_{\Gamma}{M_{A'}}(rR)\\
-\delta_{\Gamma}^*& D_{\delta_{\Gamma}}{M_{A'}}(rR)
\end{bmatrix}(I_{\Gamma}\otimes\sigma_{\hat{A}}) D_{\hat{A}_{\Gamma}}\\
=&\begin{bmatrix}
	{M_{A}}(rR)&0\\
	0& I_{\Gamma\otimes {\cld}_{*,A'}}
\end{bmatrix}\begin{bmatrix}
 D_{*,\delta_{\Gamma}}&\delta_{\Gamma}\\
-\delta_{\Gamma}^*& D_{\delta_{\Gamma}}
\end{bmatrix}\begin{bmatrix}
I_{\Gamma\otimes{\cld}_{A}}&0\\
0&{M_{A'}}(rR)
\end{bmatrix}(I_{\Gamma}\otimes\sigma_{\hat{A}}) D_{\hat{A}_{\Gamma}}.
\end{align*}
Consequently, we have
\begin{equation*}
	{M_{C,E'}(rR)D_{E'_{\Gamma}}|_{\Gamma\otimes\mathcal{H}_{\hat{A}}}}	=\hat{\gamma}_{\Gamma}(I_{\Gamma}\otimes(\sigma'_{\hat{A}})^{-1})\begin{bmatrix}
		{M_{A}}(rR)&0\\
		0& I_{\Gamma\otimes {\cld}_{*,A'}}
	\end{bmatrix}\Bigg(I_{\Gamma}\otimes \begin{bmatrix}
		 D_{*, \delta}&\delta\\
		-\delta^*& D_{\delta}
	\end{bmatrix}\Bigg)\begin{bmatrix}
		I_{\Gamma\otimes{\cld}_{A}}&0\\
		0&{M_{A'}}(rR)
	\end{bmatrix}(I_{\Gamma}\otimes\sigma_{\hat{A}}) D_{\hat{A}_{\Gamma}}.
\end{equation*}
Taking the limit $r$ goes to $1$, we get the desired result.
\end{proof}
Now we establish a similar factorization as in Proposition \ref{prop second component} for  $M_{C,E'}D_{E'_{\Gamma}}$ where the domain is restricted to $\Gamma\otimes \mathcal{H}_C$.
\begin{proposition}\label{prop firstcomponent}
		Let $\underline{C}$ be a row contraction on $\mathcal{H}_C$.	Let $\underline{E},\underline{E'},\underline{\hat{A}}, \hat{\gamma}$ and $\delta$ be as defined in Equations (\ref{equation steplifting}) and (\ref{contractions}), and Remark \ref{delta}. Then
\begin{align*}
			{M_{C,E'} D_{E'_{\Gamma}}|_{\Gamma\otimes\mathcal{H}_C}}=&\begin{bmatrix}
				 D_{*,\hat{\gamma}_{\Gamma}}&\hat{\gamma}_{\Gamma}
			\end{bmatrix}\begin{bmatrix}
			I_{\Gamma\otimes{\cld}_C}&0\\
			0&I_{\Gamma}\otimes(\sigma'_{\hat{A}})^{-1}
		\end{bmatrix}
	\begin{bmatrix}
		I_{\Gamma\otimes{\cld}_C}&0&0\\
		0&{M_{A}}&0\\
		0&0&I_{\Gamma\otimes{\cld}_{*,A'}}
	\end{bmatrix}\Bigg(I_{\Gamma}\otimes\\
&\begin{bmatrix}
	I_{{\cld}_C}&0&0\\
	0& D_{*,\delta}&\delta\\
	0&-\delta^*& D_{\delta}
\end{bmatrix}\Bigg)
\begin{bmatrix}
I_{\Gamma\otimes{\cld}_{C}}&0&0\\
0&I_{\Gamma\otimes{\cld}_{A}}&0\\
0&0&{M_{A'}}
\end{bmatrix}\begin{bmatrix}
I_{\Gamma\otimes{\cld}_{C}}&0\\
0&I_{\Gamma}\otimes\sigma_{\hat{A}}
\end{bmatrix}\\
&(I_{\Gamma}\otimes\sigma_{E'} D_{E'}|_{\mathcal{H}_C})
\end{align*} 
where $\sigma_{\hat{A}}$ and $\sigma'_{\hat{A}}$  are the unitaries defined in Lemmas \ref{lemma isomtery} and \ref{lemma adjoint isometry} for the lifting $\underline{\hat{A}}$ of $\underline{A},$ and $\sigma_{E'}$ is the unitary defined in Lemma \ref{lemma isomtery} for the lifting $\underline{E'}$ of $\underline{C}.$
\end{proposition}
\begin{proof}
	By Equation (\ref{eqn characteristic for HC}) and Lemma \ref{Lemma MAJI}, we know that
	\begin{align*}
		{M_{C,E'}(rR)D_{E'_{\Gamma}}|_{\Gamma\otimes\mathcal{H}_C}}=& D_{C_{\Gamma}}-\hat{\gamma}_{\Gamma} D_{*,\hat{A}_{\Gamma}}(I-r \underline{R}_H\underline{\hat{A}}_{\Gamma}^*)^{-1} D_{*,\hat{A}_{\Gamma}}\hat{\gamma}_{\Gamma}^* D_{C_{\Gamma}}\\
		=& D_{C_{\Gamma}}-\hat{\gamma}_{\Gamma}[I+{M_{\hat{A}}}(rR)\underline{\hat{A}}_{\Gamma}^*]\hat{\gamma}_{\Gamma}^* D_{C_{\Gamma}}\\
	=& D_{C_{\Gamma}}-\hat{\gamma}_{\Gamma}\hat{\gamma}_{\Gamma}^* D_{C_{\Gamma}}- {M_{C,E'}(rR)D_{E'_{\Gamma}}|_{\Gamma\otimes\mathcal{H}_{\hat{A}}}}\underline{\hat{A}}_{\Gamma}^*\hat{\gamma}_{\Gamma}^* D_{C_{\Gamma}}\\
		=& D_{*,\hat{\gamma}_{\Gamma}}^2 D_{C_{\Gamma}}-{M_{C,E'}(rR)D_{E'_{\Gamma}}|_{\Gamma\otimes\mathcal{H}_{\hat{A}}}}\underline{\hat{A}}_{\Gamma}^*\hat{\gamma}_{\Gamma}^* D_{C_{\Gamma}}\\
		=&\begin{bmatrix}
			 D_{*,\hat{\gamma}_{\Gamma}}&I_{\Gamma\otimes{\cld}_{C}}
		\end{bmatrix}
	\begin{bmatrix}
	I_{\Gamma\otimes{\cld}_{C}}&0\\
		0&{M_{C,E'}(rR) D_{E'_{\Gamma}}|_{\Gamma\otimes\mathcal{H}_{\hat{A}}}}
	\end{bmatrix}
\begin{bmatrix}
	 D_{*,\hat{\gamma}_{\Gamma}} D_{C_{\Gamma}}\\
	-\underline{\hat{A}}_{\Gamma}^*\hat{\gamma}_{\Gamma}^* D_{C_{\Gamma}}
\end{bmatrix}\\
=&\begin{bmatrix}
	 D_{*,\hat{\gamma}_{\Gamma}}&I_{\Gamma\otimes{\cld}_{C}}
\end{bmatrix}
\begin{bmatrix}
	I_{\Gamma\otimes{\cld}_C}&0\\
	0&{M_{C,E'}(rR)D_{E'_{\Gamma}}|_{\Gamma\otimes\mathcal{H}_{\hat{A}}}}
\end{bmatrix}(I_{\Gamma}\otimes\sigma_{E'} D_{E'}|_{\mathcal{H}_C}).
		\end{align*}
	By Proposition \ref{prop second component}, we have the following
	\begin{align*}
		{M_{C,E'}(rR)D_{E'_{\Gamma}}|_{\Gamma\otimes\mathcal{H}_C}}=&\begin{bmatrix}
			 D_{*,\hat{\gamma}_{\Gamma}}&I_{\Gamma\otimes{\cld}_{C}}
		\end{bmatrix}\begin{bmatrix}
			I_{\Gamma\otimes{\cld}_C}&{0}\\
			0&\hat{\gamma}_{\Gamma}(I_{\Gamma}\otimes(\sigma'_{\hat{A}})^{-1})
		\end{bmatrix}
		\begin{bmatrix}
			I_{\Gamma\otimes{\cld}_C}&0&0\\
			0&{M_{A}}(rR)&0\\
			0&0&I_{\Gamma\otimes{\cld}_{*,A'}}
		\end{bmatrix}\Bigg(I_{\Gamma}\otimes\\
		&\begin{bmatrix}
			I_{{\cld}_C}&0&0\\
			0& D_{*,\delta}&\delta\\
			0&-\delta^*& D_{\delta}
		\end{bmatrix}\Bigg)
		\begin{bmatrix}
			I_{\Gamma\otimes{\cld}_C}&0&0\\
			0&I_{\Gamma\otimes{\cld}_A}&0\\
			0&0&{M_{A'}}(rR)
		\end{bmatrix}\begin{bmatrix}
			I_{\Gamma\otimes{\cld}_C}&0\\
			0&{I_{\Gamma}\otimes}\sigma_{\hat{A}}
		\end{bmatrix}
			\big(I_{\Gamma}\otimes\sigma_{E'} D_{E'}|_{\mathcal{H}_C}\big)\\
			=&\begin{bmatrix}
				 D_{*,\hat{\gamma}_{\Gamma}}&\hat{\gamma}_{\Gamma}
			\end{bmatrix}\begin{bmatrix}
			I_{\Gamma\otimes{\cld}_C}&{0}\\
			0&I_{\Gamma}\otimes(\sigma'_{\hat{A}})^{-1}
		\end{bmatrix}
	\begin{bmatrix}
	I_{\Gamma\otimes{\cld}_C}&0&0\\
	0&{M_{A}}(rR)&0\\
	0&0&I_{\Gamma\otimes{\cld}_{*,A'}}
\end{bmatrix}\Bigg(I_{\Gamma}\otimes\\
&\begin{bmatrix}
I_{{\cld}_C}&0&0\\
0& D_{*,\delta}&\delta\\
0&-\delta^*& D_{\delta}
\end{bmatrix}\Bigg)
\begin{bmatrix}
I_{\Gamma\otimes{\cld}_C}&0&0\\
0&I_{\Gamma\otimes{\cld}_A}&0\\
0&0&{M_{A'}}(rR)
\end{bmatrix}\begin{bmatrix}
I_{\Gamma\otimes{\cld}_C}&0\\
0&{I_{\Gamma}\otimes}\sigma_{\hat{A}}
\end{bmatrix}
\big(I_{\Gamma}\otimes\sigma_{E'} D_{E'}|_{\mathcal{H}_C}\big).
	\end{align*}
Now, by taking the limit $r\rightarrow 1$, we get the result.
\end{proof}
Combining Proposition \ref{prop second component} and Proposition \ref{prop firstcomponent}, we obtain the following presentation for the characteristic function.
\begin{theorem} \label{fact}
		Let $\underline{C}$ be a row contraction on $\mathcal{H}_C$.	Let $\underline{E},\underline{E'},\underline{\hat{A}}, \hat{\gamma}$ and $\delta$ be as defined in Equations (\ref{equation steplifting}) and (\ref{contractions}), and Remark \ref{delta}. Then
	\begin{align*}
		{M_{C,E'} D_{E'_{\Gamma}}}=&\begin{bmatrix}
			 D_{*,\hat{\gamma}_{\Gamma}}&\hat{\gamma}_{\Gamma}
		\end{bmatrix}\begin{bmatrix}
		I_{\Gamma\otimes{\cld}_C}&{0}\\
		0&I_{\Gamma}\otimes(\sigma'_{\hat{A}})^{-1}
	\end{bmatrix}
\begin{bmatrix}
I_{\Gamma\otimes{\cld}_C}&0&0\\
0&{M_{A}}&0\\
0&0&I_{\Gamma\otimes{\cld}_{*,A'}}
\end{bmatrix}\Bigg(I_{\Gamma}\otimes\\
&\begin{bmatrix}
I_{{\cld}_C}&0&0\\
0& D_{*,\delta}&\delta\\
0&-\delta^*& D_{\delta}
\end{bmatrix}\Bigg)
\begin{bmatrix}
I_{\Gamma\otimes{\cld}_{C}}&0&0\\
0&I_{\Gamma\otimes{\cld}_{A}}&0\\
0&0&{M_{A'}}
\end{bmatrix}\begin{bmatrix}
I_{\Gamma\otimes{\cld}_{C}}&0\\
0&I_{\Gamma}\otimes\sigma_{\hat{A}}
\end{bmatrix}\\
&(I_{\Gamma}\otimes\sigma_{E'} D_{E'}).
	\end{align*}
	where $\sigma_{\hat{A}}$ and $\sigma'_{\hat{A}}$  are the unitaries defined in Lemmas \ref{lemma isomtery} and \ref{lemma adjoint isometry} for the lifting $\underline{\hat{A}}$ of $\underline{A},$ and $\sigma_{E'}$ is the unitary defined in Lemma \ref{lemma isomtery} for the lifting $\underline{E'}$ of $\underline{C}.$
\end{theorem}
\begin{remark}

	 For the symbol $\theta_{C,E'}=M_{C,E'}|_{e_0\otimes\cld_{E'}},$ the factorization of Theorem \ref{fact} is as follows
	\begin{align}\label{symbol}
	{\theta_{C,E'}}=&\begin{bmatrix}
	D_{*,\hat{\gamma}}&\hat{\gamma}_{\Gamma}
	\end{bmatrix}\begin{bmatrix}
	I_{{\cld}_C}&{0}\\
	0&I_{\Gamma}\otimes(\sigma'_{\hat{A}})^{-1}
	\end{bmatrix}
	\begin{bmatrix}
	I_{{\cld}_C}&0&0\\
	0&{\theta_{A}}&0\\
	0&0&I_{\Gamma\otimes{\cld}_{*,A'}}
	\end{bmatrix}\\
	&\begin{bmatrix}
	I_{{\cld}_C}&0&0\\
	0& D_{*,\delta}&I_{\Gamma}\otimes \delta\\
	0&-\delta^*& I_{\Gamma}\otimes D_{\delta}
	\end{bmatrix}
	\begin{bmatrix}
	I_{\cld_{C}}&0&0\\
	0&I_{{\cld}_{A}}&0\\
	0&0&{\theta_{A'}}
	\end{bmatrix}\begin{bmatrix}
	I_{{\cld}_{C}}&0\\
	0&\sigma_{\hat{A}}
	\end{bmatrix}\sigma_{E'}. \nonumber
\end{align}
\end{remark}
The Theorem \ref{fact}  is illustrated by the following example.

\begin{example}
	Let $C = \frac{1}{2},A=0$ and $A'=0$ be  contractions on the Hilbert spaces $\clh_{C}=\clh_{A}=\clh_{A'}=\mathbb{C}$ respectively. Consider the liftings $E$ of $C$ by $A$  and $E'$  of $E$ by $A'$ defined as follows
	\[E = \begin{bmatrix}
	C&0\\B&A
	\end{bmatrix} = \frac{1}{2}\begin{bmatrix}
	1 & 0\\
	1 & 0
	\end{bmatrix}\text{ and } E'= \begin{bmatrix}
	E&0\\
	B'&A'
	\end{bmatrix}
	= \frac{1}{2}\begin{bmatrix}
	1 & 0 & 0 \\
	1 & 0 & 0\\
	1 & 0 & 0 
	\end{bmatrix}\] Then $\clh_E=\mathbb{C}^2$ and $\clh_{E'}=\mathbb{C}^3$. 
	Since $E'$ is also a lifting for $C$, we have
	\[
	E' = \begin{bmatrix}
	C&0\\
	\hat{B}&\hat{A}
	\end{bmatrix},
	\text{
		where }\hat{B} = \begin{bmatrix}
	\frac{1}{2}\\
	\frac{1}{2}
	\end{bmatrix}\text{ and }
	\hat{A} = \begin{bmatrix}
	0 & 0\\
	0 & 0
	\end{bmatrix}.\] 
	Note that   $D_C = \frac{\sqrt{3}}{2}$, $D_{*, \hat{A}}=I_{\mathbb{C}^2}$  and 
	$	D_{E'} 
	= \begin{bmatrix}
	\frac{1}{2} & 0 & 0\\
	0 & 1 & 0\\
	0 & 0 & 1
	\end{bmatrix}.$
	{Using Proposition \ref{Prop lifting},  
		$({\hat{B}})^{*} = D_C\hat{\gamma} D_{*, \hat{A}}$ for some contraction $\hat{\gamma}\in B(\cld_{*,\hat{A}},\cld_{C})$. So,
		\begin{align*}
		\hspace{5pt}\begin{bmatrix}
		\frac{1}{2}&\frac{1}{2}\end{bmatrix}&= \frac{\sqrt{3}}{2}\hat{\gamma}I_{\mathbb{C}^2}
		\end{align*}
This implies 		$\hat{\gamma} =  \begin{bmatrix}
		\frac{1}{\sqrt{3}} & \frac{1}{\sqrt{3}}
		\end{bmatrix}$ and
		  $D_{*,\hat{\gamma}} = \frac{1}{\sqrt{3}}.$}
	Now using the equations (\ref{eqn dey characteristic funtion HC}) and (\ref{eqn dey characteristic function HA}), we have the following
	
	\begin{align*}
	\theta_{C,E'}(D_{E'})h_C &= \frac{1}{2 \sqrt{3}},\text{ for }h_C\in\clh_C,\\
	\theta_{C,E'}(D_{E'})h_{{\hat{A}}} &= e_1\otimes\frac{1}{\sqrt{3}}(h_A + h_{A'})\text{ for }h_{\hat{A}} = \begin{bmatrix}
	h_A\\h_{A'}
	\end{bmatrix}\in\clh_{\hat{A}}.
	\end{align*}
	Thus 	
	\[
	\theta_{C,E'}(z) = \begin{bmatrix}
	\frac{1}{\sqrt{3}} & \frac{z}{\sqrt{3}} & \frac{z}{\sqrt{3}}
	\end{bmatrix}.
	\]	
	For  $A=A'=0$, $\theta_{A}=\theta_{A'} = z.$ Moreover, { $\sigma_{\hat{A}}= \sigma'_{\hat{A}}=I_{\mathbb{C}^2}$ and 
		\begin{align*}
		\sigma_{E'}D_{E'} &= \begin{bmatrix}
		D_{*,\hat{\gamma}} D_C&0\\
		-\hat{A}^*\hat{\gamma}^* D_C& D_{\hat{A}}\end{bmatrix}\\
		&=\begin{bmatrix}
		\frac{1}{2}&0&0\\0&1&0\\0&0&1
		\end{bmatrix}=D_{E'}
		\end{align*} 
		Thus $\sigma_{E'}=I_{\mathbb{C}^3}.$}
	Therefore the right hand side of Equation  (\ref{symbol}) is 
	\begin{align*}
	&\begin{bmatrix}
	\frac{1}{\sqrt{3}}&
	\frac{1}{\sqrt{3}}&\frac{1}{\sqrt{3}}
	\end{bmatrix} \begin{bmatrix}
	1&0&0\\0&z&0\\0&0&1
	\end{bmatrix}\begin{bmatrix}
	1&0&0\\0&1&0\\0&0&z
	\end{bmatrix}\\
	&=\begin{bmatrix}
	\frac{1}{\sqrt{3}} & \frac{z}{\sqrt{3}} & \frac{z}{\sqrt{3}}
	\end{bmatrix}.
	\end{align*}
\end{example}

The following is the converse of Theorem \ref{fact}.

\begin{theorem}
	Let $\clh_C$, $\clh_A$, $\clh_{A'}$, $\clf$ and  $\clf_*$ be Hilbert spaces. Suppose $\underline{C} = (C_1,C_2,...,C_n)$, $\underline{A} = (A_1,A_2,...,A_n)$ and $\underline{A'} = (A'_1,A'_2,...,A'_n)$ are row contractions on the Hilbert spaces $\clh_C$, $\clh_A$ and $\clh_{A'}$ respectively. Let $\lambda$ be a contraction from $\cld_{A}\oplus\clf_{*}$ to $\cld_{C}$ and $U$ be a unitary map from $\clf\oplus\cld_{*,A'}$ to $\cld_{A}\oplus\clf_{*}$. Define 
	\[
	\underline{\hat{A}} := \begin{bmatrix}
		\underline{A} & 0\\
		\underline{B} & \underline{A'}
	\end{bmatrix},
	\]
	where  $\underline{B} = D_{*,A'}(P_{\cld_{*A'}}{U}^{*}|_{\cld_{A}})D_{A}$. Let
	\begin{align*}
		{\tilde{\clm}} &= \begin{bmatrix}
			{M_{A}} & 0\\
			0 & I_{\Gamma\otimes{\clf}_*}
		\end{bmatrix}
		(I_{\Gamma}\otimes {U})
		\begin{bmatrix}
			I_{\Gamma\otimes\clf} & 0\\
			0 & {M_{A'}}
		\end{bmatrix}\\
		{\clm} &= \begin{bmatrix}
			I_{\Gamma}\otimes D_{*,{\lambda}} & I_{\Gamma}\otimes{\lambda}
		\end{bmatrix} 
		\begin{bmatrix}
			I_{\Gamma\otimes\mathcal{D}_C}  & 0 & 0\\
			0 & {M_{A}} & 0\\
			0 & 0 & I_{\Gamma\otimes{\clf}_*}
		\end{bmatrix}
		\begin{bmatrix}
		I_{\Gamma\otimes\mathcal{D}_C} & 0\\
		0 & (I_{\Gamma}\otimes {U})
		\end{bmatrix}
		\begin{bmatrix}
			I_{\Gamma\otimes\mathcal{D}_C} & 0 & 0\\
			0 & I_{\Gamma\otimes\clf} & 0\\
			0 & 0 & {M_{A'}}
		\end{bmatrix}.
	\end{align*}
	If ${\tilde{\clm}}$ is a purely contractive multi-analytic operator, then there exist a unitary operator $\phi:\cld_{*,\hat{A}}\rightarrow\cld_{*,A}\oplus\clf_{*}$ such that the characteristic function ${M_{C,E'}}$ of the lifting
	$\underline{E'} = \begin{bmatrix}
		\underline{C} & 0\\
		D_{*,\hat{A}}{\phi}^{*}{\lambda}^{*}D_C & \underline{\hat{A}}
	\end{bmatrix} $   coincides with  ${\clm}$.
\end{theorem}
\begin{proof}
	Let ${U}^{*} := \begin{bmatrix}
		P & Q\\
		R & S
	\end{bmatrix},$ where
	${P}\in\clb(\cld_{A},\clf)$, ${Q}\in\clb(\clf_{*},\clf)$, ${R}\in\clb(\cld_{A},\cld_{*,A'})$ and ${S}\in\clb(\clf_{*}, \cld_{*,A'})$.
	Define $\clf^{'} := \clf \ominus {P}\cld_{A}$ and ${\clf}'_{*} := \clf_{*}\ominus {S}^{*}\cld_{*,A'} $.
	
	 We claim that ${U}\clf^{'} = {\clf}'_{*}$. Let $f'_{*}\in{\clf}'_{*}$. Then by definition of ${\clf}'_{*}$, we have 
	\[
	\langle {S}f'_{*}, D_{*,A'}h\rangle = \langle f'_{*}, {S}^{*}D_{*,A'}h\rangle = 0,\,\forall\,h\in \mathcal{H}_{A'}.
	\]
	Thus ${S}f'_{*} = 0$ and ${U}^{*}f'_{*} =  {Q}f'_{*}$. Using the fact that ${U}$ is unitary, we have
	\begin{equation}\label{eqn1}
		f'_{*} = {U}{U}^{*}f'_{*} = {Q}^{*}{Q}f'_{*}.
	\end{equation}
	If we denote $f = {Q}f'_{*}$, then by Equation (\ref{eqn1}), we have
${Q}{Q}^{*}f = {Q}{Q}^{*}{Q}f'_{*}
		= {Q}f'_{*} = f$.
	Since $U$ is a unitary operator, we have 
	$$\|f\|^{2} = \|{U}f\|^{2} = \|{P}^{*}f\|^2 + \|{Q}^*f\|^2 .$$ 
	This implies ${P}^{*}f = 0$ and consequently $f\perp {P}\cld_{A}$. Therefore, ${U}^*{\clf}'_{*}={Q}{\clf}'_{*}\subseteq {\clf}'.$
	Using similar arguments, we have 
	${U}{\clf}'\subseteq{\clf}'_{*}.$
	This proves our claim, that ${U}{\clf}' = {\clf}'_{*}$.
	
	For $f'\in{\clf}'$, we have
	\[
	{\tilde{\clm}}(e_{0}\otimes f') = \begin{bmatrix}
	{M_{A}} & 0\\
	0 & I_{\Gamma\otimes{\clf}_*}
	\end{bmatrix}
	(I_{\Gamma}\otimes {U})
	\begin{bmatrix}
	I_{\Gamma\otimes\clf} & 0\\
	0 & {M_{A'}}
	\end{bmatrix}\\ (e_0\otimes f'),
	\]
	and
	\[
	\|P_{e_0\otimes(\cld_{*,A}\oplus\clf_{*})}{\tilde{\clm}}(e_0\otimes f')\|^2 = \|e_0\otimes {U}f'\|^2 = \|f'\|^2.
	\]
	Since ${\tilde{\clm}}$ is purely contractive, we get $f' = 0$. Therefore ${\clf}' = \{0\}$ and ${\clf}'_{*} = {U}{\clf}' = \{0\}$. Hence,
$\overline{{P}\cld_{A}}  = \clf  \text{ and } \overline{{S}^{*}\cld_{*,A'}} = {\clf_{*}}.$
		For $x\in\cld_{A}$, 
		$$\|x\|^2=\|{U}^*x\|^2=\|{P}x\|^2+\|{R}x\|^2.$$ In other words, $\|x\|^2-\|{R}x\|^2=\|{P}x\|^2$, equivalently $\|D_{R}x\|=\|{P}x\|$.
	Similarly, for $y\in\cld_{*,A'}$ we have $\|D_{*,{R}}y\| = \|{S}^{*}y\|$.
	Thus we can define the isometries ${U_1}$ and ${U_2}$  by
	\[
	{U_1}({P}x) = D_{R} x \text{ and } {U_2}({S}^{*}y) = D_{*,{R}} y.
	\]
	Since $\overline{{P}\cld_{A}} = \clf$ and $\overline{{S}^{*}\cld_{*,A'}} = \clf_{*}$, we can extend the isometries ${U_1}$ and ${U_2}$  to unitary operators from $\clf$ to $\cld_{{R}}$ and $\clf_{*}$ to $\cld_{*,{R}}$, respectively. Note that $	{P}^{*} = D_{R}|_{\cld_{{R}}}{U_1}$ and ${S}^*={U_2}^*D_{*,{R}}$. Thus
	\begin{align}\label{eqn6}
		{U} = \begin{bmatrix}
			P^{*} & {R}^{*}\\
			Q^{*} & {S}^{*}
		\end{bmatrix} 
		= \begin{bmatrix}
			D_{R}|_{{\cld}_{R}}{U_1} & {R}^{*}\\
			{U_2}^{*}{Q}_1 {U_1} & {U_2}^*D_{*,{R}}
		\end{bmatrix} =
		\begin{bmatrix}I_{\cld_{A}} & 0\\0 & {U_2}^{*} \end{bmatrix}
		\begin{bmatrix}D_{R}|_{{\cld}_{R}} & {R}^{*}\\ {Q}_1 & D_{*,{R}}\end{bmatrix}
		\begin{bmatrix}{U_1} & 0\\ 0 & I_{\cld_{*,A'}}\end{bmatrix},
	\end{align}
	where ${Q}_1 = {U_2}{Q}^{*}{U_1}^{*}\in\clb(\cld_{{R}},\cld_{*,{R}})$. In Equation (\ref{eqn6}), ${U},\,{U_2},\,{U_1}$ are unitaries, which implies the middle matrix is also unitary
	\[\begin{bmatrix}
		D_{R}^2+{Q}_1^*{Q}_1&D_{R}{R}^*+{Q}_1^*D_{*,{R}}\\
		{R}D_{R}+D_{*,{R}}{Q}_1&{R}{R}^*+D_{*,{R}}^2
	\end{bmatrix}=\begin{bmatrix}
	I_{\mathcal{D}_{R}}&0\\
	0&I_{\mathcal{D}_{*,{R}}}
\end{bmatrix}.\]
Therefore, ${R}D_{R}+D_{*,{R}}{Q}_1=0$. Equivalently, $D_{*,{R}}({R}+{Q}_1)=0$ or $\overline{\text{range}}({R}+{Q}_1)\subseteq\text{Null}(D_{*,{R}})$. But ${Q}_1(\mathcal{D}_{R})\subseteq\mathcal{D}_{*,{R}}$ and ${R}(\mathcal{D}_{R})\subseteq\mathcal{D}_{*,{R}}$, thus $\overline{\text{range}}({Q}_1+{R}|_{\mathcal{D}_{R}})\subseteq\mathcal{D}_{*,{R}}=(\text{Null}(D_{*,{R}}))^{\perp}$. Consequently ${Q}_1=-{R}|_{\mathcal{D}_{R}}.$ \\
Let 
	\[
	\textit{u} := \begin{bmatrix}{U_1} & 0\\ 0 & I_{\cld_{*,A'}}\end{bmatrix},\,
	J := \begin{bmatrix}D_{R}|_{{\cld}_{R}} & {R}^{*}\\ -{R}{|_{{\cld}_{R}}} & D_{*,{R}}\end{bmatrix}\text{ and }
	\textit{v} := \begin{bmatrix}I_{\cld_{A}} & 0\\0 & {U_2} \end{bmatrix},
	\]
	then ${U} = {\textit{v}}^{*}J\textit{u}$. Also, take
	\[
	\textit{u}' = \begin{bmatrix}{U_1} & 0\\ 0 & I_{\cld_{A}}\end{bmatrix}\;\text{and}\;
	\textit{v}' = \begin{bmatrix}I_{\cld_{*,A}} & 0\\0 & {U_2} \end{bmatrix}.
	\]
	Then
	\begin{align}
		(I_{\Gamma}\otimes\textit{v}')
		\begin{bmatrix}
		{M_{A}} & 0\\
		0 & I_{\Gamma\otimes{\clf}_*}
		\end{bmatrix} =& \begin{bmatrix}
		{M_{A}} & 0\\
		0 & I_{\Gamma\otimes\mathcal{D}_{*,{R}}}
		\end{bmatrix}(I_{\Gamma}\otimes\textit{v}), \text{ and }\\
		(I_{\Gamma}\otimes\textit{u})
		\begin{bmatrix}
		I_{\Gamma\otimes\clf} & 0\\
		0 & {M_{A'}}
		\end{bmatrix} =& \begin{bmatrix}
		I_{\Gamma\otimes\mathcal{D}_{R}} & 0\\
		0 & {M_{A'}}
		\end{bmatrix}(I_{\Gamma}\otimes\textit{u}').
	\end{align}
We have
	\begin{align*}
		(I_{\Gamma}\otimes\textit{v}'){\tilde{\clm}}(I_{\Gamma}\otimes{\textit{u}}'^{*})& = (I_{\Gamma}\otimes\textit{v}')
		\begin{bmatrix}
		{M_{A}} & 0\\
		0 & I_{\Gamma\otimes{\clf}_*}
		\end{bmatrix}(I_{\Gamma}\otimes {U})\begin{bmatrix}
		I_{\Gamma\otimes\clf} & 0\\
		0 & {M_{A'}}
		\end{bmatrix}(I_{\Gamma}\otimes{\textit{u}}'^{*})\\
		&= \begin{bmatrix}
		{M_{A}} & 0\\
		0 & I_{\Gamma\otimes\mathcal{D}_{*,{R}}}
		\end{bmatrix}(I_{\Gamma}\otimes\textit{v})(I_{\Gamma}\otimes {U})(I_{\Gamma}\otimes\textit{u}^*)\begin{bmatrix}
		I_{\Gamma\otimes\mathcal{D}_{R}} & 0\\
		0 & {M_{A'}}
		\end{bmatrix}\\
	&=\begin{bmatrix}
	{M_{A}} & 0\\
	0 & I_{\Gamma\otimes\mathcal{D}_{*,{R}}}
	\end{bmatrix}(I_{\Gamma}\otimes J)\begin{bmatrix}
	I_{\Gamma\otimes\mathcal{D}_{R}} & 0\\
	0 & {M_{A'}}
	\end{bmatrix}.
	\end{align*}
From this, we have
	\begin{align*}
		{\clm} =& \begin{bmatrix}
			I_{\Gamma}\otimes D_{*,{\lambda}} & I_{\Gamma}\otimes{\lambda}
		\end{bmatrix} 
		\begin{bmatrix}
			I_{\Gamma\otimes\mathcal{D}_C} & 0\\
			0 & {\tilde{\clm}}
		\end{bmatrix}\\
	=&\begin{bmatrix}
		I_{\Gamma}\otimes D_{*,{\lambda}} & I_{\Gamma}\otimes{\lambda}
	\end{bmatrix} \begin{bmatrix}
		I_{\Gamma\otimes\mathcal{D}_C}&0\\
		0&I_{\Gamma}\otimes {v'}^*
	\end{bmatrix}
		\begin{bmatrix}
			I_{\Gamma\otimes\mathcal{D}_C}&0&0\\
			0&{M_{A}}&0\\
			0&0&I_{\Gamma\otimes\mathcal{D}_{*,{R}}}
		\end{bmatrix}\begin{bmatrix}
		I_{\Gamma\otimes\mathcal{D}_C}&0\\
		0&I_{\Gamma}\otimes J
	\end{bmatrix}
		\begin{bmatrix}
			I_{\Gamma\otimes\mathcal{D}_C} & 0&0\\
			0&I_{\Gamma\otimes\mathcal{D}_{R}}&0\\
			0 & 0&{M_{A'}}
		\end{bmatrix}\\
	&		\begin{bmatrix}
			I_{\Gamma\otimes\mathcal{D}_C} & 0\\
			0 & I_{\Gamma}\otimes\textit{u}'
		\end{bmatrix}\\
	=&\begin{bmatrix}
		I_{\Gamma}\otimes D_{*,{\lambda}} & I_{\Gamma}\otimes{\lambda} {v'}^*\sigma'_{\hat{A}}
	\end{bmatrix} \begin{bmatrix}
		I_{\Gamma\otimes\mathcal{D}_C}&0\\
		0&I_{\Gamma}\otimes (\sigma'_{\hat{A}})^{-1}
	\end{bmatrix}
	\begin{bmatrix}
		I_{\Gamma\otimes\mathcal{D}_C}&0&0\\
		0&{M_{A}}&0\\
		0&0&I_{\Gamma\otimes\mathcal{D}_{*,{R}}}
	\end{bmatrix}\begin{bmatrix}
		I_{\Gamma\otimes\mathcal{D}_C}&0\\
		0&I_{\Gamma}\otimes J
	\end{bmatrix}\\
	&\begin{bmatrix}
		I_{\Gamma\otimes\mathcal{D}_C} & 0&0\\
		0&I_{\Gamma\otimes\mathcal{D}_{R}}&0\\
		0 & 0&{M_{A'}}
	\end{bmatrix}
\begin{bmatrix}
	I_{\Gamma\otimes\mathcal{D}_C}&0\\
	0&I_{\Gamma}\otimes \sigma_{\hat{A}}
\end{bmatrix}(I_{\Gamma}\otimes\sigma_{E'})\left(I_{\Gamma}\otimes(\sigma_{E'})^{-1}\right)
	\begin{bmatrix}
		I_{\Gamma\otimes\mathcal{D}_C} & 0\\
		0 & I_{\Gamma}\otimes(\sigma_{\hat{A}})^{-1}\textit{u}'
	\end{bmatrix}\\
		 =& {M_{C,E'}}\left(I_{\Gamma}\otimes\sigma_{E'}^{-1}\right)\begin{bmatrix}
			I_{\Gamma\otimes\mathcal{D}_C} & 0\\
			0 & I_{\Gamma}\otimes(\sigma_{\hat{A}})^{-1}\textit{u}'
		\end{bmatrix},
	\end{align*}
where \begin{align*}\underline{E'}&=\begin{bmatrix}
	\underline{C}&0\\
	D_{*,\hat{A}}\phi^*{\lambda}^*D_C&\underline{\hat{A}}
\end{bmatrix},\,{\underline{\hat{A}}=\begin{bmatrix}
	\underline{A}&0\\D_{*A'}R^*D_A&\underline{A}'
	\end{bmatrix}},\end{align*}
and $\phi:={v'}^*\sigma'_{\hat{A}}:\mathcal{D}_{*,\hat{A}}\rightarrow \mathcal{D}_{*,A}\oplus\mathcal{F}_*.$
Hence ${\clm}$ coincides with ${M_{C,E'}}$.
\end{proof}

\section{Characteristic function of an iterated lifting}

In this section, we investigate the factorization properties of the characteristic function of  the minimal part of an iterated lifting. Let $\underline{E}$ be a minimal contractive lifting of a row contraction $\underline{C}$ and
$\underline{E}'$ be a minimal contractive lifting of $\underline{E}$. Define
\begin{equation}\label{eqn minimal cont}
	\clh_{\tilde{E}} := \overline{span}\{{E'_{\alpha}}x: x\in\clh_C,\,\alpha\in\tilde{\Lambda}\} \;\;\text{and}\;\; \tilde{\underline{E}} = {\underline{E}}'|_{\clh_{\tilde{E}}}.
\end{equation}
Consider $\clh_{\tilde{A}} = \clh_{E'}\ominus\clh_{\tilde{E}}$, and define
$\underline{\tilde{X}} = P_{\clh_{\tilde{E}}}\underline{E}'|_{\clh_{\tilde{A}}}$ and 
$\underline{\tilde{A}} = P_{\clh_{\tilde{A}}}\underline{E}'|_{\clh_{\tilde{A}}}$. Then
\[
\underline{E}' = \begin{bmatrix}
	\underline{\tilde{E}} & \underline{\tilde{X}}\\
	\underline{0} & \underline{\tilde{A}} 
\end{bmatrix}.
\]
As $\underline{E}'$ is a row contraction, we have
\begin{equation} \label{gamma}
\tilde{\underline{X}} = D_{*,\tilde{E}}\gamma D_{\tilde{A}}
\end{equation}
for some contraction $ \gamma\in\clb(\cld_{\tilde{A}},\cld_{*,\tilde{E}})$. 
By Nagy-Foias (see \cite{NF}) there exist a unitary 
$\sigma: \cld_{E'} \rightarrow \cld_{\tilde{E}}\oplus\cld_{\gamma}$, such that
\begin{equation} \label{sigma}
\sigma D_{E'} = \begin{bmatrix}
	D_{\tilde{E}} & -\tilde{E}^{*}\gamma D_{\tilde{A}}\\
	0 & D_{\gamma}D_{\tilde{A}}
\end{bmatrix}.
\end{equation}


\begin{theorem}\label{theorem product of characteristic function}
	Let 
$\underline{E}$ be a minimal contractive lifting of a row contraction $\underline{C}$ and
$\underline{E}'$ be a minimal contractive lifting of $\underline{E}$. If $\clh_{\tilde{E}}$ and $\underline{\tilde{E}}$ are as defined in Equation \ref{eqn minimal cont},
then $\tilde{\underline{E}}$ is a minimal contractive lifting of $\underline{C}$ and $M_{C,\tilde{E}} =    M_{C,E}M_{E,E'}(I_{\Gamma}\otimes\sigma^{-1}|_{\cld_{\tilde{E}}})$.  
\end{theorem}
\begin{proof}
From the definition of $\mathcal{H}_{\tilde{E}}$, it is easy to see that $\mathcal{H}_C\subseteq\mathcal{H}_{\tilde{E}}$ and $\underline{\tilde{E}}$ is a lifting of $\underline{C}$. Also,
\[\overline{span}\{\tilde{E}_{\alpha}x: x\in\clh_C\} = \overline{span}\{{E'_{\alpha}}x: x\in\clh_C\} = \clh_{\tilde{E}}.\]
Thus $\tilde{\underline{E}}$ is a minimal contractive lifting of $\underline{C}$.

By Definition \ref{def 2.4}, we know that the characteristic functions $M_{C,E}$ and $M_{E,E'}$ are given by
\[
M_{C,E} = P_{\Gamma\otimes\cld_C} W_1|_{\Gamma\otimes\cld_{E}}\text{ and }
M_{E,E'} = P_{\Gamma\otimes\cld_E} W_2|_{\Gamma\otimes\cld_{E'}},
\]
where
$W_1 : \hat{\clh}_{E}\rightarrow\hat{\clh}_{C}\oplus \clk_1$ and $W_2 :\hat{\clh}_{E'}\rightarrow\hat{\clh}_{E}\oplus \clk_2$ are unitary operators, as in Equation (\ref{char uni}), satisfying
\begin{align}\label{eqn7}
	W_1{V_i}^{E} = (V^{C}\oplus Y_1)_i W_1,\, W_1|_{\clh_C} = I_{\mathcal{H}_C} \text{ and }W_2{V_i}^{E'} = (V^{E}\oplus Y_2)_i W_2,\, W_2|_{\clh_E} = I_{\mathcal{H}_E},
\end{align}
 for $i=1,2,\ldots,d.$

Now, we define $Z := (W_1\oplus I_{\clk_2})W_2:{\hat{\clh}_{E'}}\rightarrow\hat{\clh}_C\oplus \clk_1\oplus \clk_2$. Note that $Z$ is a unitary operator, as it is a composition of two unitary operators. By Equation (\ref{eqn7}), we have
\[Z|_{\clh_C} = I_{\mathcal{H}_C}\text{ and }
Z{V_i}^{E'} = (V^C\oplus Y_1\oplus Y_2)_{i}Z\;\;\text{for all} \; i = 1,2,..,d.
\]
It is easy to see that
\begin{equation}\label{eqn8}
	M_{C,E}M_{E,E'} = P_{\Gamma\otimes\cld_C}Z|_{\Gamma\otimes\cld_{E'}}.
\end{equation}

Let the unitary 
$\sigma: \cld_{E'} \rightarrow \cld_{\tilde{E}}\oplus\cld_{\gamma}$ be as defined in Equation (\ref{sigma}).
Thus $Z(\Gamma\otimes \cld_{E'}) = (W_1\otimes I_{\clk_2})W_2(I_{\Gamma}\otimes\sigma^{-1})((\Gamma\otimes \cld_{\tilde{E}})\oplus(\Gamma\otimes \cld_{\gamma})).$

We claim that $P_{\Gamma\otimes\cld_C}Z(\Gamma\otimes\cld_{E'}) 
= P_{\Gamma\otimes\cld_C}Z(I_{\Gamma}\otimes\sigma^{-1})(\Gamma\otimes\cld_{\tilde{E}})$. To verify the claim, it is enough to show that $Z(I_{\Gamma}\otimes\sigma^{-1})(e_0\otimes \cld_{\gamma})$ is orthogonal to $e_0\otimes \cld_C$. For $h_{A}\in\clh_{A}$ and $h_{C}\in\clh_{C}$,
\begin{align*}
	&\langle Z(I_{\Gamma}\otimes\sigma^{-1})(e_0\otimes D_{\gamma}h_A), (e_0\otimes (D_C)_i h_C)\rangle\\ 
	&= \langle W_2(I_{\Gamma}\otimes\sigma^{-1})(e_0\otimes D_{\gamma}h_A), W_1^{*}(e_0\otimes (D_{C})_ih_C)\rangle\\
	&= \langle W_2(I_{\Gamma}\otimes\sigma^{-1})(e_0\otimes D_{\gamma}h_A), W_1^{*}({V_i}^C h_C - C_i h_C)\rangle\\
	&= \langle W_2(I_{\Gamma}\otimes\sigma^{-1})(e_0\otimes D_{\gamma}h_A), {V_i}^E W_1^{*}h_C - W_1^{*}C_i h_C)\rangle\\
	&= \langle W_2(I_{\Gamma}\otimes\sigma^{-1})(e_0\otimes D_{\gamma}h_A), {V_i}^E h_C - C_i h_C\rangle\\
	&= \langle (I_{\Gamma}\otimes\sigma^{-1})(e_0\otimes D_{\gamma}h_A), W_2^{*}({V_i}^E h_C - C_i h_C)\rangle\\
	&= \langle (I_{\Gamma}\otimes\sigma^{-1})(e_0\otimes D_{\gamma}h_A), {V_i}^{E'} W_2^{*}h_C - W_2^{*}C_i h_C)\rangle\\
	&= \langle (I_{\Gamma}\otimes\sigma^{-1})(e_0\otimes D_{\gamma}h_A), {V_i}^{E'} h_C - {C_i} h_C)\rangle\\
	&= \langle e_0\otimes \sigma^{-1} D_{\gamma}h_A, {E_i}'h_C - C_i h_C + e_0\otimes (D_{E'})_ih_C\rangle\\
	&= \langle D_{\gamma}h_A, \sigma (D_{E'})_ih_C\rangle\\
	&= \langle D_{\gamma}h_A, (D_{\tilde{E}})_ih_C\rangle= 0.
\end{align*}
Consider the unitary operator $Z_1: \hat{\clh}_{\tilde{E}}\rightarrow\hat{\clh}_{C}\oplus \clk_1\oplus \clk_2$ defined by 
\[
Z_1 := Z(I_{\mathcal{H}_{\tilde{E}}}\oplus (I_{\Gamma}\otimes\sigma^{-1}|_{\cld_{\tilde{E}}})).
\]  
For $h_{\tilde{E}} \oplus \sum_{\alpha} e_{\alpha} \otimes D_{\tilde{E}}h_{\alpha}\in\hat{H}_{\tilde{E}}$, we have
\begin{align*}
	Z_1 {V_i}^{\tilde{E}} \left(h_{\tilde{E}} \oplus \sum_{\alpha} e_{\alpha} \otimes D_{\tilde{E}}h_{\alpha}\right)
	&= Z_1\left(\tilde{E}_i h_{\tilde{E}} \oplus [e_0\otimes D_{\tilde{E}i}h_{\tilde{E}} + e_i\otimes\sum_{\alpha}e_{\alpha}\otimes D_{\tilde{E}}h_{\alpha}]\right)\\
	&= Z(I_{\clh_{\tilde{E}}} \oplus (I_{\Gamma}\otimes\sigma^{-1}))\left(\tilde{E}_i h_{\tilde{E}} \oplus [e_0\otimes D_{\tilde{E}i}h_{\tilde{E}} + e_i\otimes\sum_{\alpha}e_{\alpha}\otimes D_{\tilde{E}}h_{\alpha}]\right)\\
	&= Z\left(E'_i h_{\tilde{E}} \oplus[ e_0\otimes D_{E'i}h_{\tilde{E}} + e_i\otimes\sum_{\alpha}e_{\alpha}\otimes D_{E'}h_{\alpha}]\right)\\
	&= Z{V_i}^{E'}\left(h_{\tilde{E}} \oplus \sum_{\alpha}e_{\alpha}\otimes D_{E'}h_{\alpha}\right)\\
	&= (V^C\oplus Y_1\oplus Y_2)_i Z\left(h_{\tilde{E}} \oplus \sum_{\alpha}e_{\alpha}\otimes D_{E'}h_{\alpha}\right)\\
	&= (V^C\oplus Y_1\oplus Y_2)_i Z\left(I_{\mathcal{H}_{\tilde{E}}}\oplus \left(I_{\Gamma}\otimes\sigma^{-1}\right)\right)\left(h_{\tilde{E}} \oplus \sum_{\alpha}e_{\alpha}\otimes D_{\tilde{E}}h_{\alpha}\right)\\
	&= (V^C\oplus Y_1\oplus Y_2)_i Z_1 \left(h_{\tilde{E}} \oplus \sum_{\alpha}e_{\alpha}\otimes D_{\tilde{E}}h_{\alpha}\right).     
\end{align*}
Thus we obtained
\[
Z_{1}{V_i}^{\tilde{E}} = (V^C\oplus Y_1\oplus Y_2)_i Z_1\;\;\text{for all}\;i = 1,2,...,d
\]
Also, we have $Z_1|_{\mathcal{H}_C}=I$. By definition \ref{def 2.4}, the characteristic function for the minimal lifting $\underline{\tilde{E}}$ of $\underline{C}$ is 
\[
M_{C,\tilde{E}} = P_{\Gamma\otimes\cld_C}Z_1|_{\Gamma\otimes\cld_{\tilde{E}}}.
\] 
From Equation (\ref{eqn8}), we have
\begin{align*}
	M_{C,\tilde{E}} &= P_{\Gamma\otimes\cld_C}Z(I_{\Gamma}\otimes\sigma^{-1}|_{\cld_{\tilde{E}}})\\
	&= P_{\Gamma\otimes\cld_C}Z|_{\Gamma\otimes\cld_{E'}} (I_{\Gamma}\otimes\sigma^{-1}|_{\cld_{\tilde{E}}}) \\
	& = M_{C,E}M_{E,E'} (I_{\Gamma}\otimes\sigma^{-1}|_{\cld_{\tilde{E}}}).
\end{align*}
This completes the proof.
\end{proof}
The following example validates the Theorem \ref{theorem product of characteristic function}.
\begin{example}
	Let $C,\, A,\,A' = 0$ be zero operators on $\clh_C,\,\clh_A,\,\clh_{A'}=\mathbb{C}$, respectively. Consider the lifting $E$ of $C$ and $E'$ of $E$ defined by
	\[E = \begin{bmatrix}
	C&0\\B&A
	\end{bmatrix} = \frac{1}{\sqrt{2}}\begin{bmatrix}
	0 & 0\\
	1 & 0
	\end{bmatrix}\text{ and } E'= \begin{bmatrix}
	E&0\\
	B'&A'
	\end{bmatrix}
	= \frac{1}{\sqrt{2}}\begin{bmatrix}
	0 & 0 & 0 \\
	1 & 0 & 0\\
	1 & 0 & 0 
	\end{bmatrix}\] on $\clh_E=\mathbb{C}^2$ and $\clh_{E'}=\mathbb{C}^3$, respectively.  
	Then $E'$ is also a lifting for $C$ and
	\[
	E' = \begin{bmatrix}
	C&0\\
	B'_C&A'_C
	\end{bmatrix},
	\text{
		where }B'_C = \begin{bmatrix}
	\frac{1}{\sqrt{2}}\\
	\frac{1}{\sqrt{2}}
	\end{bmatrix}\text{ and }
	A'_C = \begin{bmatrix}
	0 & 0\\
	0 & 0
	\end{bmatrix}.\] It is easy to see that $E$ is a minimal contractive lifting of $C$ and $E'$ is a minimal contractive lifting of $E$. We know
	\[
	\overline{\text{span}}\{{E'}^nx: x\in\clh_{C} \;\text{and}\;n\geq0\} = \text{span}\left\{(1,0,0),\left(0,\frac{1}{\sqrt{2}},\frac{1}{\sqrt{2}}\right)\right\}\neq\clh_{E'}.	
	\]
	Thus $E'$ is not a minimal contractive lifting of $C$. 
	
	Also note that   $D_C = I_{\mathbb{C}}$, $D_{*, A} = I_{\mathbb{C}}$, $D_{*, A'}=I_{\mathbb{C}}$ and 
	\[
	{D_E} = \sqrt{I_{\clh_E} -\underline{E}^{\ast}\underline{E}}\\
	= \begin{bmatrix}
	\frac{1}{\sqrt{2}} & 0\\
	0 & 1
	\end{bmatrix},\,D_{E'} = \sqrt{I_{\clh_{E'}} - E'^{\ast}E'}
	= \begin{bmatrix}
	0 & 0 & 0\\
	0 & 1 & 0\\
	0 & 0 & 1
	\end{bmatrix}.
	\] 
	We have  $B^* = {D_C}\gamma D_{*, A}$ and ${B'}^{\ast} = {D_E}\gamma' D_{*, A'}$, which implies $\gamma = \frac{1}{\sqrt{2}}$ and $\gamma' = \begin{bmatrix}
	1 \\ 0
	\end{bmatrix}$. 
	
	By Equations (\ref{eqn dey characteristic funtion HC}) and (\ref{eqn dey characteristic function HA}), we have the following for $h_C\in\clh_C$ and $h_A\in\clh_A$:
	\begin{align*}
	\theta_{C,E}(D_E)h_C = e_{0}\otimes\frac{h_C}{2}\text{ and }
	\theta_{C,E}(D_E)h_A = e_1\otimes\frac{h_A}{\sqrt{2}}.
	\end{align*}
	Similarly for $h_E = \begin{bmatrix}
	h_{C}\\ h_{A}
	\end{bmatrix} \in \clh_E$ and $h_{A'}\in\clh_{A'}$, we have
	\begin{align*}
	\theta_{E,E'}(D_{E'})h_E = e_{0}\otimes
	\begin{bmatrix}
	0\\ h_{A}
	\end{bmatrix}\text{ and }
	\theta_{E,E'}(D_{E'})h_{A'} = e_1\otimes\begin{bmatrix}
	h_{A'} \\ 0
	\end{bmatrix}.
	\end{align*}  
	As we have already mentioned, $E'$ is a contractive lifting for $C$, thus  
	$({B'_C})^{*} = D_C\gamma'_C D_{*, A_C}$. This implies $\gamma'_C =  \begin{bmatrix}
	\frac{1}{\sqrt{2}} & \frac{1}{\sqrt{2}}
	\end{bmatrix}$ and 
	\begin{align*}
	\theta_{C,E'}(D_{E'})h_C &= 0,\text{ for }h_C\in\clh_C,\\
	\theta_{C,E'}(D_{E'})h_{{A'_C}} &= e_1\otimes\frac{1}{\sqrt{2}}(h_A + h_{A'})\text{ for }h_{{A'_C}} = \begin{bmatrix}
	h_A\\h_{A'}
	\end{bmatrix}\in\clh_{A'_C}.
	\end{align*}
	We can also represent $\theta_{C,E}$, $\theta_{E,E'}$ and $\theta_{C,E'}$ in the following matrix form:
	\[
	\theta_{C,E}(z) = \begin{bmatrix}
	\frac{1}{\sqrt{2}} & \frac{z}{\sqrt{2}}
	\end{bmatrix},\,
	\theta_{E,E'}(z) = \begin{bmatrix}
	0 & 0 & z\\
	0 & 1 & 0
	\end{bmatrix},\,
	\theta_{C,E}\theta_{E,E'}(z) = \begin{bmatrix}
	0 & \frac{z}{\sqrt{2}} & \frac{z}{\sqrt{2}}
	\end{bmatrix}.
	\]
	{Now consider,	\[\clh_{\tilde{E}} = 
		\overline{\text{span}}\{{E'}^nx: x\in\clh_{C} \;\text{and}\;n\geq0\} = \text{span}\left\{(1,0,0),\left(0,\frac{1}{\sqrt{2}},\frac{1}{\sqrt{2}}\right)\right\}.	
		\]
		Then $\clh_{\tilde{E}} = \clh_{C}\oplus\clh_{\tilde{A}}$ where $\clh_{\tilde{A}} =\text{span }\left\{(0,\frac{1}{\sqrt{2}},\frac{1}{\sqrt{2}})\right\}.$ To write $E'$ in terms of the basis $$\{(1,0,0)(0,\frac{1}{\sqrt{2}},\frac{1}{\sqrt{2}})(0,- \frac{1}{\sqrt{2}},\frac{1}{\sqrt{2}})\},$$
	consider the unitary operator
		\[
	P = 
	\begin{bmatrix}
	1&0& 0\\
	0 & \frac{1}{\sqrt{2}} & - \frac{1}{\sqrt{2}}\\
	0 & \frac{1}{\sqrt{2}} & \frac{1}{\sqrt{2}} 
	\end{bmatrix},
	\]
	and obtain 
	\[
	P^{-1} E' P = 
	\begin{bmatrix}
	0 & 0 & 0\\
	1 & 0 & 0\\
	0 & 0 & 0 
	\end{bmatrix}
	\]
	}
	Since $\clh_{\tilde{E}}$ is a reducing subspace of $\clh_{E'}$ w.r.t. $\underline{E'},$ we identify 
	\begin{align*}
	\tilde{E} = E'|_{\clh_{\tilde{E}}} =
	\begin{bmatrix}
	C&0\\\tilde{B}&\tilde{A}
	\end{bmatrix}
	\end{align*}
	with
	$\begin{bmatrix}
	0&0\\
	1&0
	\end{bmatrix}.$
	We have $D_{*,\tilde{A}} = I_{\mathbb{C}},$ 
	\[ D_{\tilde{E}}=\begin{bmatrix}
	0&0\\
	0&1
	\end{bmatrix} \mbox{~and~} \cld_{\tilde{E}}=\{(0,k): k \in \mathbb{C}\}.\]
	Since $\tilde{B}^*=D_C\tilde{\gamma}D_{*,\tilde{A}}$, that implies $\tilde{\gamma} = 1$. Again using Equations (\ref{eqn dey characteristic funtion HC}) and (\ref{eqn characteristic for HA}), we get
	
	\begin{align*}
	\theta_{C,\tilde{E}}(D_{\tilde{E}})h_C &= 0,\text{ and }
	\theta_{C,\tilde{E}}(D_{\tilde{E}})h_{\tilde{A}} = e_1\otimes h_{\tilde{A}} \text{ for } h_C\in\mathcal{H}_C,\,h_{\tilde{A}}\in\mathcal{H}_{\tilde{A}}.
	\end{align*}
	In other words,
	\begin{align*}
	\theta_{C,\tilde{E}}(z) &= \begin{bmatrix}
	0 & z
	\end{bmatrix}.
	\end{align*}
	We also have,
	\begin{align*}
	\theta_{C,E}\theta_{E,E'} \begin{bmatrix}
	1&0& 0\\
	0 & \frac{1}{\sqrt{2}}  & - \frac{1}{\sqrt{2}}\\
	0 & \frac{1}{\sqrt{2}} & \frac{1}{\sqrt{2}} 
	\end{bmatrix}|_{\cld_{\tilde{E}}}   &= \begin{bmatrix}
	0 & z
	\end{bmatrix}. 
	\end{align*}
		Thus $$\theta_{C,\tilde{E}} = \theta_{C,E}\theta_{E,E'}\sigma^{-1}|_{\cld_{\tilde{E}}}, $$
	where $\sigma^{-1}= P$.
\end{example}

\section*{Acknowledgment}
	The first author was supported by the Department of Mathematics, IIT Bombay, India. The second author was supported by SERB MATRICS Grant number MTR/2018/000343.

\end{document}